\documentclass[a4paper,twoside,12pt]{article}

\usepackage{amssymb,amsfonts,amsmath,amsthm,latexsym}
\usepackage{hyperref}
\usepackage[hmargin=1.2in,vmargin=1.2in]{geometry}
\usepackage{graphicx,caption,subcaption}
\usepackage[auth-sc-lg,affil-sl]{authblk}
\newtheorem{thm}{Theorem}[section]

\newtheorem{defn}[thm]{Definition}

\newtheorem{prob}[thm]{Problem}
\newtheorem{prop}[thm]{Proposition}

\numberwithin{equation}{section}

\def\ni{\noindent}
\def\cC{\mathcal{C}}

\title{\textbf{\sc Certain Chromatic Sums of Some Cycle Related Graph Classes}}

\author{N. K. Sudev}
\affil{\small Department of Mathematics\\ Vidya Academy of Science \& Technology \\ Thalakkottukara, Thrissur - 680501, India.\\ {\tt sudevnk@gmail.com}}

\author{K. P. Chithra}
\affil{\small Naduvath Mana, Nandikkara \\ Thrissur - 680301, India.\\ {\tt chithrasudev@gmail.com}}

\author{Johan Kok}
\affil{\small Tshwane Metropolitan Police Department\\ City of Tshwane, Republic of South Africa \\ {\tt kokkiek2@tshwane.gov.za}}

\date{}

\begin{document}
\maketitle

\begin{abstract}
Let $\mathcal{C} = \{c_1,c_2, c_3, \ldots,c_k\}$ be a certain type of proper $k$-colouring of a given graph $G$ and $\theta(c_i)$ denote the number of times a particular colour $c_i$ is assigned to the vertices of $G$. Then, the colouring sum of  a given graph $G$ with respect to the colouring $\cC$, denoted by $\omega_{\cC}(G)$, is defined to be $\omega(\cC) = \sum\limits_{i=1}^{k}i\,\theta(c_i)$. The colouring sums such as $\chi$-chromatic sum, $\chi^+$-chromatic sum, $b$-chromatic sum, $b^+$-chromatic sum etc. are some of these types of colouring sums that have been studied recently. Motivated by these studies on certain chromatic sums of graphs, in this paper, we study certain chromatic sums for some standard cycle related graphs. 
\end{abstract}

\ni \textbf{Keywords:} Chromatic number, $\chi$-chromatic sum, $\chi^+$-chromatic sum, $b$-chromatic number, $\varphi'$-chromatic sum, $\varphi^+$-chromatic sum.

\ni \textbf{AMS Classification Numbers:} 05C15, 05C38.

\section{Introduction}

For general notations and concepts in the theory of graphs and digraphs, we refer to \cite{BM1,CZ1,GY1,FH,DBW}. Unless mentioned otherwise, all graphs mentioned in this paper are non-trivial, simple, connected, finite and undirected graphs.

The notions of certain types of colourings of given graphs attracted much research interest. In a proper colouring of a graph $G$, no two adjacent vertices in $G$ are coloured using the same colour. The minimum number of colours in a proper colouring of $G$ is called the \textit{chromatic number} of $G$, denoted by $\chi(G)$.

For a \textit{proper $k$-colouring} $\mathcal{C} = \{c_1,c_2,c_3, \ldots,c_k\}$ of a graph $G$, the set of vertices $V_{c_i} = \{v_j: v_j \mapsto c_i, v_j \in V(G), c_i \in \cC\}$, $1 \le i \le k$ is called a colour class of the colour $c_i$. It is to be noted that $V(G) = \bigcup\limits_{i=1}^{k}V_{c_i}$. 

Recently, some studies on the concepts of certain chromatic sums of graphs have been appeared in the literature. The concepts of the general colouring sum of graphs with respect to different types of graph colourings can be defined as follows.

\begin{defn}\label{Defn-1.1}{\rm 
\cite{KSC1,ES1} Let $\mathcal{C} = \{c_1,c_2, c_3, \ldots,c_k\}$ allows a certain type of proper $k$-colouring of a given graph $G$ and $\theta(c_i)$ denotes the number of times a particular colour $c_i$ is assigned to vertices of $G$. Then, the \textit{colouring sum} of a colouring $\cC$ of a given graph $G$, denoted by $\omega(\cC)$, is defined to be $\omega(\cC) = \sum\limits_{i=1}^{k}i\,\theta(c_i)$.}
\end{defn}

Invoking the above definition, the following two parameters have been introduced in \cite{KSC1}. 

\begin{defn}\label{Defn-1.2}{\rm 
\cite{KSC1,ES1} The \textit{$\chi$-chromatic sum} of a graph $G$, denoted by $\chi'(G)$, is defined as $\chi'(G) = \min\{\sum\limits_{i=1}^{k}i\,\theta(c_i)\}$ and the \textit{$\chi^+$-chromatic sum} of a graph $G$, denoted by $\chi^+(G)$, is defined as $\chi^+(G)=\max\{\sum\limits_{i=1}^{k}i\,\theta(c_i)\}$, where the sum varies over all minimum proper colourings of $G$.
}\end{defn}

The \textit{$b$-chromatic number} of a graph $G$ of order $n$ is defined as the maximum number $k$ of colours that can be used to colour the vertices of $G$, such that we obtain a proper colouring and each colour $i$, with $1\le i\le k$, has at least one representant $x_i$ adjacent to a vertex of every colour $j$, $1\le j\ne i \le k$. Such a colouring is called a \textit{$b$-colouring} of $G$ (see \cite{EK1,IM1}). The concept of $b$-chromatic number has attracted much attention (see \cite{EK1,IM1,KS1,VI1,VS1,VS2,VS3,VV1}). 

\begin{defn}\label{Defn-1.3}{\rm
\cite{KSC1,LS1}. The \textit{$b$-chromatic sum} of a graph $G$, denoted by $\varphi'(G)$, is defined as $\min\{\sum\limits_{i=1}^{k} i\,\theta(c_i)\}$ and the \textit{$b^+$-chromatic sum} of a graph $G$, denoted by $\varphi^+(G)$, is defined as $\max\{\sum\limits_{i=1}^{k} i\,\theta(c_i)\}$, where the sum varies over a $b$-coloring using $\varphi(G)$ colors.}
\end{defn}

The concepts of different chromatic sums are reported to be extremely important and can be applied in many related studies on certain networks like transportation networks, network flows, domination theory, resource allocation etc. (see\cite{BST,EK2,ES1,LS1}). 

\section{New Directions}

The $\chi$-chromatic sum and $\chi^+$-chromatic sum of paths, cycles, wheel graphs, complete graphs and complete bipartite graphs have been determined in \cite{KSC1} and $b$-chromatic sums of those graphs have been determined in \cite{LS1}. $b^+$-chromatic sum of these graph classes have also been studied in \cite{KSC1}. Motivated from these studies, in this paper, we discuss different colouring sums of some other particular classes of cycle related graphs. 

\subsection{Colouring Sums of Double Wheel Graphs}

A \textit{wheel graph} $W_{n+1}$ is a graph obtained by joining every vertex of a cycle $C_n$ to a single external vertex. This external vertex is said to be the central vertex of the wheel graph. The $b$-chromatic sum of wheel graphs have been determined in \cite{LS1} and their $b^+$-chromatic sum, $\chi$-chromatic sum and $\chi^+$-chromatic sum  have been determined in \cite{KSC1}. Let us now proceed to determine these parameters for some graphs generated from wheel graphs.

First, recall that a \textit{double wheel graph}, denoted by $DW_n$, is the graph obtained by joining all vertices of a disjoint union of two cycles $C_n$ to an external vertex. That is, $DW_n=2C_n+k_1$. The following result discusses the $\chi$-chromatic sum of double wheel graphs.

\begin{prop}\label{Prop-2.1}
The $\chi$-chromatic sum of a double wheel graph $DW_n$ is given by
\begin{equation*}
\chi'(DW_n)=
\begin{cases}
3(n+1); & \text{if $n$ is even},\\
3n+7; & \text{if $n$ is odd}.
\end{cases}
\end{equation*}
\end{prop}
\begin{proof}
Let $G=DW_n=2C_n+K_1$. Let $\{v_1,v_2,v_3,\ldots,v_n\}$ and $\{u_1,u_2,u_3\ldots,u_n\}$ be the vertex set of the two cycles in $G$ and let $v$ be its central vertex. Here we have to consider the following cases.
\begin{enumerate}\itemsep0mm
\item[(i)] Let $n$ be an even integer. Then, a minimal proper colouring of $G$ contains $3$ colours, say $c_1,c_2,c_3$. Let this $3$-colouring of $G$ be in such a way that the corresponding colour classes of $G$ are $V_{c_1}=\{v_1,v_3,v_5,\ldots, v_{n-1}, u_1,u_3,u_5,\ldots, u_{n-1}\}$, $V_{c_2}=\{v_2,v_4,v_6,\ldots, v_n, u_2,u_4,u_6,\ldots, u_n\}$ and $V_{c_3}=\{v\}$. Then, $\theta(c_1)=\theta(c_2)=n$ and $\theta(c_3)=1$. Therefore, $\chi'(G)=1\,(n)+2\,(n)+3\times 1=3(n+1)$. 

\item[(ii)] Let $n$ be an odd integer. Then, a minimal proper colouring of $G$ contains $4$ colours, say $c_1,c_2,c_3, c_4$. Colour the vertices of $G$ in such a way that the corresponding colour classes are $V_{c_1}=\{v_1,v_3,v_5,\ldots, v_{n-2}, u_1,u_3,u_5,\ldots, u_{n-2}\}$, $V_{c_2}=\{v_2,v_4,v_6,\ldots, v_{n-1}, u_2,u_4,u_6,\ldots, u_{n-1}\}$, $V_{c_3}=\{v_n,u_n\}$ and $V_{c_4}=\{v\}$. Then, $\theta(c_1)=\theta(c_2)=n-1$ and $\theta(c_3)=2$ and $\theta(c_4)=1$. Therefore, $\chi'(G)=1\,(n-1)+2\,(n-1)+3\times 2+4\times 1=3n+7$. 
\end{enumerate}
\vspace{-0.5cm}
This completes the proof.
\end{proof}

\ni The next result describes the $\chi^+$-sum of double wheel graphs.

\begin{prop}\label{Prop-2.2}
The $\chi^+$-chromatic sum of a double wheel graph $DW_n$ is given by
\begin{equation*}
\chi^+(DW_n)=
\begin{cases}
5n+1; & \text{if $n$ is even},\\
7n-2; & \text{if $n$ is odd}.
\end{cases}
\end{equation*}
\end{prop}
\begin{proof}
Then, we have the following cases.

\begin{enumerate}\itemsep0mm
\item[(i)] Let $n$ be an even integer. Let the $3$-colouring of $G$, mentioned in the proof of Proposition \ref{Prop-2.1}, be in such a way that the corresponding colour classes of $G$ are $V_{c_1}=\{v\}$, $V_{c_2}=\{v_2,v_4,v_6,\ldots, v_n, u_2,u_4,u_6,\ldots, u_n\}$ and $V_{c_3}=\{v_1,v_3,v_5,\ldots, v_{n-1},u_1,u_3,u_5,\ldots, u_{n-1}\}$. Then, $\theta(c_1)=1$ and $\theta(c_2)=\theta(c_3)=n$ and  Therefore, $\chi^+(G)=3\,(n)+2\,(n)+1\times 1=5n+1$. 

\item[(ii)] Let $n$ be an odd integer. Then,in a $4$-proper colouring of $G$, interchange the colour $c_1$ with the colour $c_4$ and colour $c_2$  with colour $c_3$ in the colouring of graphs mentioned in Proposition \ref{Prop-2.1}. Therefore, the corresponding colour classes are $V_{c_1}=\{v\}$, $V_{c_2}=\{v_n,u_n\}$,  $V_{c_3}=\{v_2,v_4,v_6,\ldots, v_{n-1}, u_2,u_4,u_6,\ldots,\\ u_{n-1}\}$ and $V_{c_4}=\{v_1,v_3,v_5,\ldots, v_{n-2}, u_1,u_3,u_5,\ldots, u_{n-2}\}$. Then, $\theta(c_1)=1$, $\theta(c_2)=2$ and $\theta(c_3)=\theta(c_4)=n-1$. Therefore, $\chi^+(G)=4\,(n-1)+3\,(n-1)+2\times 2+1\times 1=7n-2$. 
\end{enumerate}
\vspace{-0.5cm}
This completes the proof.
\end{proof}

\ni In the next theorem, we discuss the $b$-chromatic sum of a double wheel graph $DW_n$.

\begin{thm}\label{Thm-2.3}
The $b$-chromatic sum of a double wheel graph $DW_n$ is given by
\begin{equation*}
\varphi'(DW_n)=
\begin{cases}
15; & \text{if $n=4$},\\
3n+10; & \text{if $n\ne 4$ is even},\\
3n+7; & \text{if $n$ is odd}.
\end{cases}
\end{equation*}
\end{thm}
\begin{proof}
It has already been known that a $b$-coloring of the wheel graph $C_4+K_1$ contains $3$ colours and for all other wheel graphs a $b$-coloring requires $4$ colours (see \cite{LS1}). We also note that the same colouring pattern used for the cycle in a wheel graph $C_n+K_1$ can be used for both cycles in a double wheel graph $2C_n+K_1$. Hence, a $b$- colouring of $DW_n$ contains three colours if $n=4$ and four colours if $n\ne 4$. 

First assume that $n=4$. Let $\{v_1,v_2,v_3,v_4\}$ and $\{u_1,u_2,u_3,u_4\}$ be the vertex sets of two cycles and $v$ be the central vertex of the graph $DW_4$. Now, assume that $\cC=\{c_1,c_2,c_3\}$ be a $b$-colouring for $DW_4$. We can colour the vertices of $DW_4$ in such a way that we get the colour classes as $V_{c_1}=\{v_1,v_3,u_1,u_3\}, V_{c_2}=\{v_2,v_4,u_2,u_4\}$ and $V_{c_3}=\{v\}$. That is, $\theta(c_1)=\theta(c_2)=4, \theta(c_3)=1$. Hence, $\varphi'(G)=1\times 4+2\times 4+3\times 1=15$.

\ni Next, assume that $n\ne 4$. Then, we need to consider the following cases.

\begin{enumerate}\itemsep0mm
\item[(i)] Let $n$ be even and $\cC= \{c_1,c_2,c_3,c_4\}$ be a $b$-colouring of $DW_n$. Then, we can colour the vertices of $DW_n$ in such a way that we get the corresponding colour classes as $V_{c_1}=\{v_1,v_4,v_6,\ldots,v_{n-2},u_1,u_4,u_6,\ldots,u_{n-2}\}, V_{c_2}=\{v_2,v_5,v_7,\ldots,v_{n-1},u_2,u_5,u_7,\ldots,u_{n-1}\}, V_{c_3}=\{v_3,v_n,u_3,u_n\}$ and $V_{c_4}=\{v\}$. Hence, $\theta(c_1)=\theta(c_2)=n-2, \theta(c_3)=4$ and $\theta(c_4)=1$. Therefore, $\varphi'(G)=1\,(n-2)+2\,(n-2)+3\times 4+4\times 1= 3n+10$.

\item[(ii)]  Let $n$ be odd and $\cC= \{c_1,c_2,c_3,c_4\}$ be a $b$-colouring of $DW_n$. Then, we can colour the vertices of $DW_n$ in such a way that we get the corresponding colour classes as $V_{c_1}=\{v_1,v_4,v_6,\ldots,v_{n-1},u_1,u_4,u_6,\ldots,u_{n-1}\}, V_{c_2}=\{v_2,v_5,v_7,\ldots,v_n,u_2,u_5,u_7,\ldots,u_n\}, V_{c_3}=\{v_3,u_3\}$ and $V_{c_4}=\{v\}$.  Hence, $\theta(c_1)=\theta(c_2)=n-1, \theta(c_3)=2$ and $\theta(c_4)=1$. Therefore, $\varphi'(G)=1\,(n-1)+2\,(n-1)+3\times 2+4\times 1= 3n+7$.
\end{enumerate}
\vspace{-0.5cm}
This complete the proof.
\end{proof}

The $b^+$-chromatic sum can be determined by reversing the colouring pattern mentioned in Theorem \ref{Thm-2.3}. The next result determines the $b^+$-chromatic sum of double wheel graphs.

\begin{thm}\label{Thm-2.4}
The $b^+$-chromatic sum of a double wheel graph $DW_n$ is given by
\begin{equation*}
\varphi^+(DW_n)=
\begin{cases}
21; & \text{if $n=4$},\\
7n-5; & \text{if $n\ne 4$ is even},\\
7n-2; & \text{if $n$ is odd}.
\end{cases}
\end{equation*}
\end{thm}
\begin{proof}
First assume that $n=4$ and assume that $\cC=\{c_1,c_2,c_3\}$ be a $b$-colouring for $DW_4$. We can colour the vertices of $DW_4$ in such a way that we get the colour classes as $V_{c_1}=\{v\}, V_{c_2}=\{v_1,v_3,u_1,u_3\}$ and $V_{c_3}=\{v_2,v_4,u_2,u_4\}$. Hence, $\theta(c_1)=1$ and  $\theta(c_2)=\theta(c_3)=4$. Hence, $\varphi'(G)=1\times 1+2\times 4+3\times 4=21$.

Next, assume $n\ne 4$ and is even. Then we can find a $b$-colouring $\cC= \{c_1,c_2,c_3,c_4\}$ of $DW_n$ such that the corresponding colour are $V_{c_1}=\{v\}, V_{c_2}=\{v_3,v_n,u_3,u_n\}, V_{c_3}=\{v_1,v_4,v_6,\ldots,v_{n-2},u_1,u_4,u_6,\ldots,u_{n-2}\}$ and $V_{c_4}=\{v_2,v_5,v_7,\ldots,v_{n-1},u_2,u_5,u_7,\ldots,\\u_{n-1}\}$.  Hence, $\theta(c_1)=1, \theta(c_2)=4, \theta(c_3)=\theta(c_4)=n-2$. Therefore, $\varphi'(G)=1\times 1+2\times 4+3\,(n-2)+4\,(n-2)= 7n-5$.

Next, assume $n$ and is odd. Then we can find a $b$-colouring $\cC= \{c_1,c_2,c_3,c_4\}$ of $DW_n$ such that the corresponding colour are $V_{c_1}=\{v\}, V_{c_2}=\{v_3,u_3\}, V_{c_3}=\{v_1,v_4,v_6,\ldots,v_{n-1},u_1,u_4,u_6,\ldots,u_{n-1}\}, V_{c_4}=\{v_2,v_5,v_7,\ldots,v_{n},u_2,u_5,u_7,\ldots,u_{n}\}$.  Hence, $\theta(c_1)=1, \theta(c_2)=2, \theta(c_3)=\theta(c_4)=n-1$. Therefore, $\varphi'(G)=1\times 1+2\times 2+3\,(n-1)+4\,(n-1)= 7n-2$.
\end{proof}

\subsection{Colouring Sums of Helm Graphs}

Another well known graph related to cycles and wheels is a \textit{Helm graph}, which is obtained by attaching pendant vertices to each vertex of the cycle of a wheel graph. A Helm graph is denoted by $H_n$. Let us now discuss the various colouring sums of Helm graphs.

\begin{thm}\label{Thm-2.5}
The $\chi$-chromatic sum of a Helm graph $H_n$ is given by
\begin{equation*}
\chi'(H_n)=
\begin{cases}
3(n+1); & \text{if $n$ is even},\\
3n+7; & \text{if $n$ is odd}.
\end{cases}
\end{equation*}
\end{thm}
\begin{proof}
We know that the minimal proper colouring of a Helm graph $H_n$ contains $3$ colours if $n$ is even and contains $4$ colours if $n$ is odd. Let $v$ be the central vertex, $\{v_1,v_2,v_3,\ldots,v_n\}$ be the vertex set of the the cycle $C_n$ and $\{u_1,u_2,u_3,\ldots,u_n\}$ be the set of pendant vertices in $H_n$. Now consider the following cases.

\begin{enumerate}\itemsep0mm
\item[(i)] Assume that $n$ is even. Then, let $\cC=\{c_1,c_2,c_3\}$ be a minimal proper colouring of $H_n$. Here, we have the following possibilities.
\begin{enumerate}
\item The pendant vertices of $H_n$ can be coloured using the same colours used for colouring the vertices of the cycle $C_n$. This colouring can be done in such a way that the colouring classes are given by $V_{c_1}=\{v_1,v_3,v_5,\ldots,v_{n-1}, u_2,u_4,u_6,\ldots, u_n\}, V_{c_2}=\{v_2,v_4,v_6,\ldots, v_n,u_1,u_3,u_5,\\ \ldots,u_{n-1}\}$ and $V_{c_3}=\{v\}$. Here, $\theta(c_1)=\theta(c_2)=n$ and $\theta(c_3)=1$ and hence the colouring sum with respect to this colouring is $1(n)+2(n)+3\times1=3(n+1)$.

\item If the pendant vertices are coloured using the colour of the central vertex, we get the minimal proper colouring such that the colour classes are $V_{c_1}=\{v,u_1,u_2,u_3\ldots,u_n\}, V_{c_2}=\{v_1,v_3,v_5\ldots,v_{n-1}\}$ and $V_{c_3}=\{v_2,v_4,v_6\ldots,v_{n}\}$. Therefore, $\theta(c_1)=n+1$ and $\theta(c_2)=\theta(c_3)=\frac{n}{2}$ and the corresponding coloring sum is $1(n+1)+2\frac{n}{2}+3\frac{n}{2}=\frac{7n}{2}+1$.
\end{enumerate} 
\vspace{-0.25cm}
It can be noted that colouring the pendant vertices using all three colours in any manner, will not result in a colouring sum less than $\min\{3(n+1),\frac{7n}{2}+1\}=3(n+1)$. Therefore, $\chi'(G)=3(n+1)$ for all even $n\ge 4$. 

\item[(ii)]Assume that $n$ is odd. Then, let $\cC=\{c_1,c_2,c_3,c_4\}$ be a minimal proper colouring of $H_n$. Here, we have the following possibilities.
\begin{enumerate}
\item As in the above case, we can colour the pendant vertices of $H_n$ using the colours used for colouring the vertices of the cycle $C_n$ so that the colour classes we get are  $V_{c_1}=\{v_1,v_3,v_5,\ldots,v_{n-2}, u_2,u_4,u_6,\ldots, u_{n-1}\}, V_{c_2}=\{v_2,v_4,v_6,\ldots, v_{n-1},u_1,u_3,u_5,\ldots,u_{n-2}\}$, $V_{c_3}=\{v_n,u_1\}$ and $V_{c_4}=\{v\}$. Here, $\theta(c_1)=\theta(c_2)=n-1$, $\theta(c_3)=2$ and $\theta(c_4)=1$ and hence the colouring sum with respect to this colouring is $1(n-1)+2(n-1)+3\times 2+4\times 1=3n+7$.

\item The pendant vertices of $H_n$ can also be coloured using the colour of the central vertex, so that we get the minimal proper colouring such that the colour classes are $V_{c_1}=\{v,u_1,u_2,u_3\ldots,u_n\}, V_{c_2}=\{v_1,v_3,v_5\ldots,v_{n-2}\},\\ V_{c_3}=\{v_2,v_4,v_6\ldots,v_{n-1}\}$ and $V_{c_4}=\{v_n\}$. Therefore, $\theta(c_1)=n+1$ and $\theta(c_2)=\theta(c_3)=\frac{n-1}{2}$ and $\theta(c_4)=1$ and hence the corresponding coloring sum is $1(n+1)+2\frac{n-1}{2}+3\frac{n-1}{2}+4\times 1=\frac{7(n+1)}{2}-1$.
\end{enumerate}
\vspace{-0.25cm}
As mentioned in the previous case, it can be noted that colouring the pendant vertices using all three colours in any manner, will not result in a colouring sum less than $\min\{3n+7,\frac{7(n+1)}{2}-1\}=3n+7$. Therefore, $\chi'(G)=3n+7$ for all odd $n\ge 3$.
\end{enumerate}
\vspace{-0.5cm}
This completes the proof.
\end{proof}

\ni The following theorem discusses the $\chi^+$-chromatic sum of helm graphs 

\begin{thm}\label{Thm-2.6}
The $\chi^+$-chromatic sum of a Helm graph $H_n$ is given by
\begin{equation*}
\chi^+(H_n)=
\begin{cases}
5n+1; & \text{if $n$ is even},\\
7n-2; & \text{if $n$ is odd}.
\end{cases}
\end{equation*}
\end{thm}
\begin{proof}
We can determine the $\chi^+$-chromatic sum of $H_n$ by reversing the colouring patterns. Consider the following cases.
\begin{enumerate}\itemsep0mm
\item[(i)] Let $n$ be even and $\cC=\{c_1,c_2,c_3\}$ be a minimal proper colouring of $H_n$. Here, we have the following possibilities.
\begin{enumerate}\itemsep0mm 
\item The pendant vertices of $H_n$ can be coloured using the same colours used for colouring the vertices of the cycle $C_n$. This colouring can be done in such a way that the colouring classes are given by $V_{c_1}=\{v\}, V_{c_2}=\{v_1,v_3,v_5,\ldots,v_{n-1}, u_2,u_4,u_6,\ldots, u_n\}$ and $V_{c_3}=\{v_2,v_4,v_6,\ldots, v_n,u_1,u_3,\\ u_5, \ldots,u_{n-1}\}$. Here, $\theta(c_1)=1$ and $\theta(c_2)=\theta(c_3)=n$. Hence, the colouring sum with respect to this colouring is $1\times 1+2(n)+3(1)=5n+1$.

\item We can colour the pendant vertices using the colour of the central vertex, in such a way that the corresponding colour classes of $H_n$ are given by $V_{c_1}= \{v_1,v_3,v_5\ldots,v_{n-1}\}, V_{c_2}=\{v_2,v_4,v_6\ldots,v_{n}\}$ and $V_{c_3}=\{v,u_1,u_2,u_3\ldots,u_n\}$. Therefore, $\theta(c_1)=\theta(c_2)=\frac{n}{2}$ and $\theta(c_3)=n+1$ and the corresponding coloring sum is $1\frac{n}{2}+2\frac{n}{2}+3(n+1)=\frac{9n}{2}+2$.
\end{enumerate}
\vspace{-0.25cm}

Note that colouring the pendant vertices using all three colours in any manner, will not result in a colouring sum greater than than $\max\{5n+1,\frac{9n}{2}+2\}=5n+1$. Therefore, $\chi^+(G)=5n+1$ for all even $n\ge 4$. 

\item[(ii)]Assume that $n$ is odd and let $\cC=\{c_1,c_2,c_3,c_4\}$ be a minimal proper colouring of $H_n$. Here, we have the following possibilities.
\begin{enumerate}
\item As mentioned in the above case, it is possible to colour the pendant vertices of $H_n$ using the colours used for colouring the vertices of the cycle $C_n$ in such a way that the colour classes thus we get are  $V_{c_1}=\{v\}, V_{c_2}=\{v_n,u_1\}, V_{c_3}=\{v_1,v_3,v_5,\ldots,v_{n-2}, u_2,u_4,u_6,\ldots, u_{n-1}\}$ and $V_{c_4}=\{v_2,v_4,v_6,\ldots, v_{n-1},u_1,u_3,u_5,\ldots,u_{n-2}\}$. Here, $\theta(c_1)=1,\theta(c_2)=2$, $\theta(c_3)=\theta(c_4)=n-1$. Hence, the colouring sum with respect to this colouring is $1\times 1+2\times 2+3(n-1)+4(n-1)=7n-2$.

\item The pendant vertices of $H_n$ can also be coloured using the colour of the central vertex in such a way that we get the colour classes $V_{c_1}=\{v_n\}, V_{c_2}=\{v_1,v_3,v_5\ldots,v_{n-2}\}, V_{c_3}=\{v_2,v_4,v_6\ldots,v_{n-1}\}$ and $V_{c_4}=\{v,u_1,u_2,u_3\ldots,u_n\}$. Therefore, $\theta(c_1)=1, \theta(c_2)=\theta(c_3)=\frac{n-1}{2}$ and $\theta(c_4)=n+1$. Hence, the corresponding coloring sum is $1\times 1+2\frac{n-1}{2}+3\frac{n-1}{2}+4(n+1)=\frac{13n+5}{2}$.
\end{enumerate}
\vspace{-0.25cm}

Here also, it can be noted that colouring the pendant vertices using all four colours in any manner, will not result in a colouring sum greater than $\max\{7n-2,\frac{13n+5}{2}\}=7n-2$. Therefore, $\chi'(G)=7n-2$ for all odd $n\ge 3$.
\end{enumerate}
%\vspace{-0.25cm}
This completes the proof.
\end{proof}

\ni The $b$-chromatic numbers of Helm graphs have been  determined in \cite{VS3}. The results proved in this paper, provide us a good background to determine $b$-chromatic and $b^+$-chromatic sums of Helm graphs. The following theorem describes the $b$-chromatic sum of Helm graphs.

\begin{thm}\label{Thm-2.7}
The $b$-chromatic sum of a Helm graph $H_n$ is given by
\begin{equation*}
\varphi'(H_n)=
\begin{cases}
14 & \text{if $n=3$},\\
25 & \text{if $n=4$},\\
21 & \text{if $n=5$},\\
30 & \text{if $n=6$},\\
3n+13 & \text{if $n\ge7$}.
\end{cases}
\end{equation*} 
\end{thm}
\begin{proof}
Let $v$ be the central vertex, $\{v_1,v_2,v_3,\ldots,v_n\}$ be the set of vertices in the cycle $C_n$ and $\{u_1,u_2,u_3,\ldots,u_n\}$ be the set of pendant vertices in $H_n$. Then we have to consider the following cases.

\begin{enumerate}\itemsep0mm
\item[(i)] First assume that $n=3$. It is known that a $b$-coloring of the Helm graph $H_3$ contains $4$ colours (see \cite{VS3}). Let this colouring be $\cC=\{c_1,c_2,c_3,c_4\}$. Colour the vertices of $H_3$ in such a way that the corresponding colour classes we get are $V_{c_1}=\{v_1,u_2,u_3\}$, $V_{c_2}=\{v_2,u_1\}$, $V_{c_3}=\{v_3\}$ and $V_{c_4}=\{v\}$. This colouring is illustrated in Figure \ref{fig:Fig-1}. Therefore, $\theta(c_1)=3,\theta(c_2)=2$ and $\theta(c_3)=\theta(c_4)=1$. Hence, $\varphi'(H_3)=1\times 3+2\times 2+3\times 1+4\times1=14$.

\begin{figure}[h!]
\centering
\includegraphics[width=0.65\linewidth]{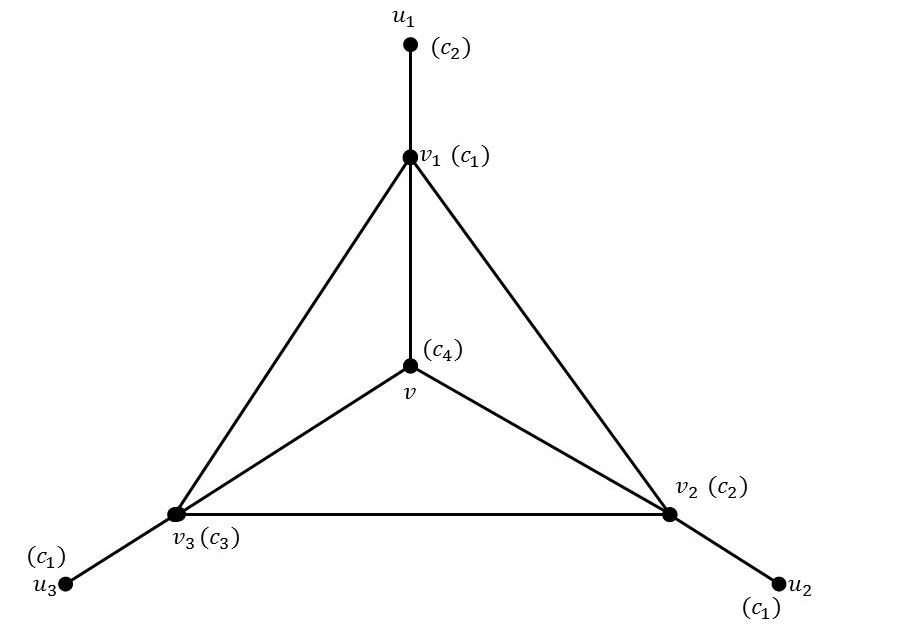}
\caption{\small A $b$-colouring of $H_3$ so that the $b$-chromatic sum is minimum.}
\label{fig:Fig-1}
\end{figure}

\item[(ii)] Next, assume that $n=4$. The $b$-colouring of $H_4$ contains $5$ colours (see \cite{VS3}). Let this colouring be $\cC=\{c_1,c_2,c_3,c_4,c_5\}$. We can colour the vertices so that the corresponding colour classes are $V_{c_1}=\{v_1,u_3\}$, $V_{c_2}=\{v_2,u_4\}$, $V_{c_3}=\{v_3,u_1\}$, $V_{c_4}=\{v_4,u_2\}$ and $V_{c_5}=\{v\}$. This colouring is illustrated in Figure \ref{fig:Fig-2}. Therefore, $\theta(c_1)=\theta(c_2)=\theta(c_3)=\theta(c_4)=2$ and $\theta(c_5)=1$. Hence, $\varphi'(H_4)=1\times 2+2\times 2+3\times 2+4\times 2+5\times 1=25$.

\begin{figure}[h!]
\centering
\includegraphics[width=0.6\linewidth]{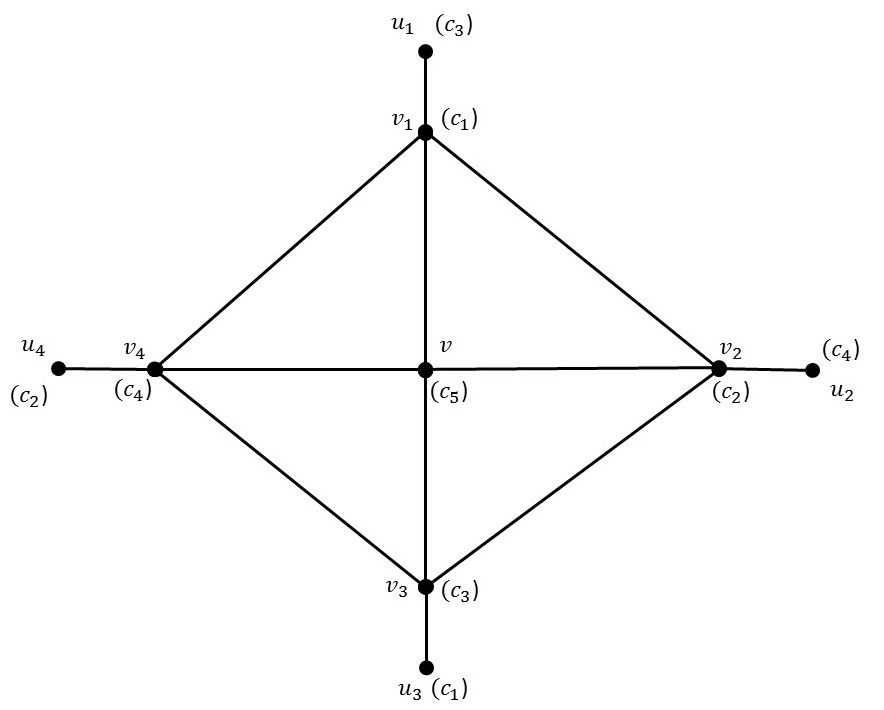}
\caption{\small A $b$-colouring of $H_4$ so that the $b$-chromatic sum is minimum.}
\label{fig:Fig-2}
\end{figure}

\item[(iii)] Now, assume that $n=5$.  A $b$-colouring of $H_5$ contains $4$ colours (see \cite{VS3}). Let $\cC=\{c_1,c_2,c_3,c_4\}$ be such a $b$-colouring of $H_5$. We can colour the vertices of $H_5$ in such a way that the corresponding colour classes are $V_{c_1}=\{v_1,v_4,u_2,u_3,u_5\}$, $V_{c_2}=\{v_2,u_1,u_4\}$, $V_{c_3}=\{v_3,v_5\}$,  and $V_{c_4}=\{v\}$. This colouring is illustrated in Figure \ref{fig:Fig-3}. Therefore, $\theta(c_1)=5, \theta(c_2)=3, \theta(c_3)=2$ and $\theta(c_4)=1$. Hence, $\varphi'(H_5)=1\times 5+2\times 3+3\times 2+4\times 1=21$. 

\begin{figure}[h!]
\centering
\includegraphics[width=0.7\linewidth]{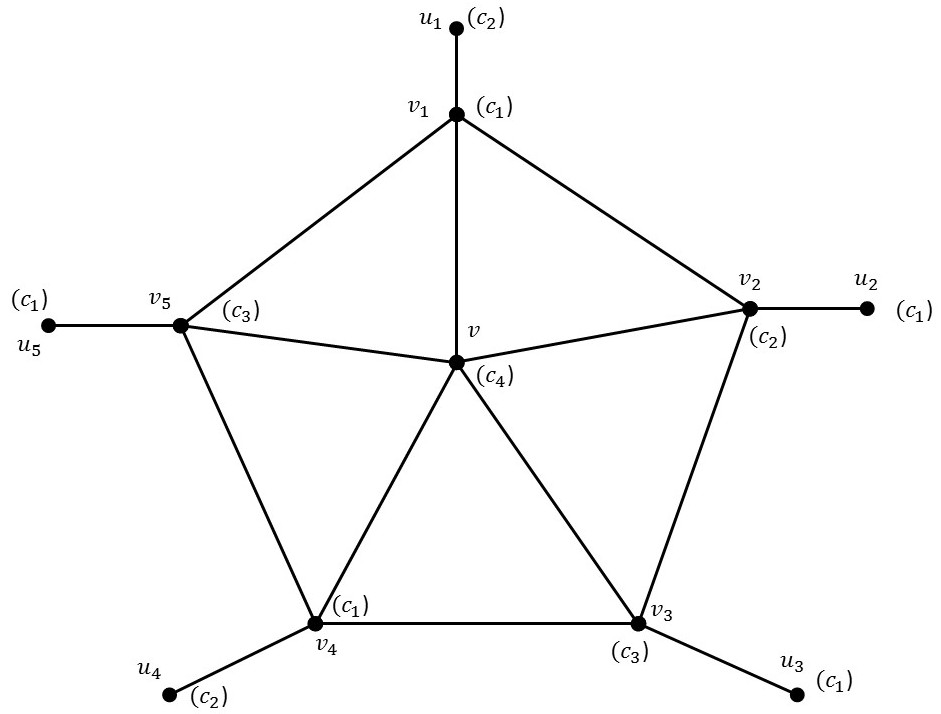}
\caption{\small A $b$-colouring of $H_5$ so that the $b$-chromatic is minimum.}
\label{fig:Fig-3}
\end{figure}

\item[(iv)] Assume that $n=6$. A $b$-colouring of $H_6$ contains $5$ colours (see \cite{VS3}). Let $\cC=\{c_1,c_2,c_3,c_4,c_5\}$ be the required $b$-colouring of $H_6$. Here, we can colour the vertices of $H_6$ in such a way that the corresponding colour classes are $V_{c_1}=\{v_1,u_3,u_4,u_5,u_6\}$, $V_{c_2}=\{v_4,v_6,u_2\}$, $V_{c_3}=\{v_2,u_5\}$, $V_{c_4}=\{v_3,u_1\}$ and $V_{c_5}=\{v\}$. This colouring is illustrated in Figure \ref{fig:Fig-4}. 
Therefore, $\theta(c_1)=5,\theta(c_2)=3, \theta(c_3)=\theta(c_4)=2$ and $\theta(c_5)=1$. Hence, $\varphi'(H_6)=1\times 5+2\times 3+3\times 2+4\times 2+5\times 1=30$.

\begin{figure}[h!]
\centering
\includegraphics[width=0.6\linewidth]{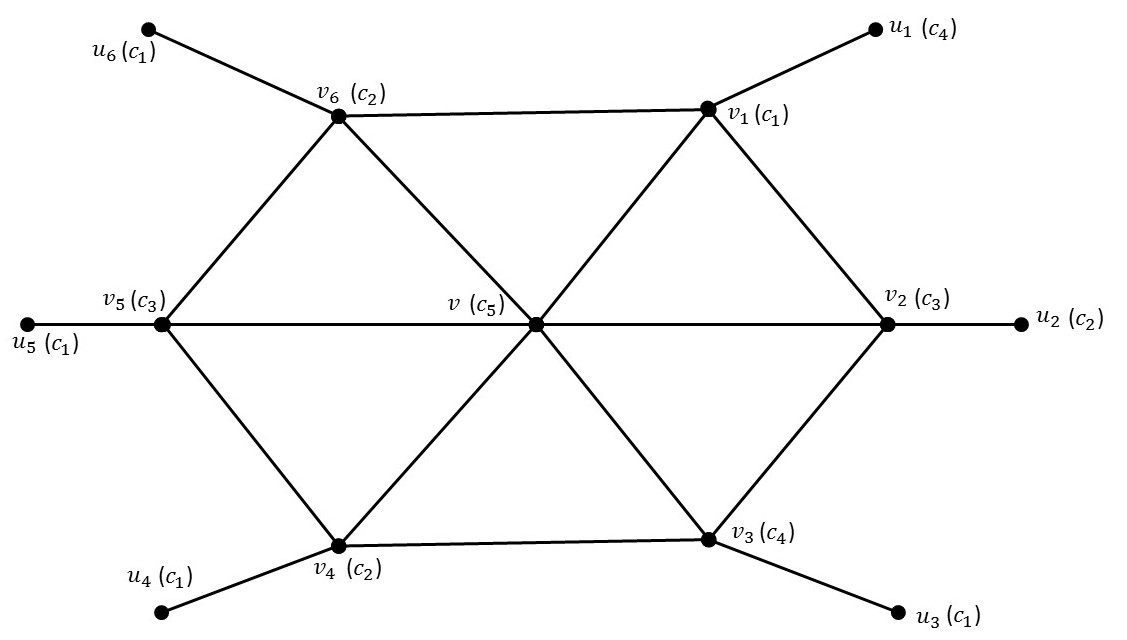}
\caption{\small A $b$-colouring of $H_6$ so that the $b$-chromatic sum is minimum.}
\label{fig:Fig-4}
\end{figure}

\item[(v)] Assume that $n\ge 7$. A $b$-coloring of a Helm graph $H_n;n\ge 7$ contains $5$ colours, say $c_1,c_2,c_3,c_4,c_5$. Then, we need to consider the following cases. 
\begin{enumerate}\itemsep0mm
\item Let $n$ be even. Now, we can colour the vertices of $H_n$ such that we get the colour classes $V_{c_1}=\{v_1,v_6,v_8,\ldots,v_{n-2},u_3,u_5,u_7,\ldots, u_{n-1},u_4,u_n\}$, $V_{c_2}=\{v_2, v_7,v_9,v_{11},\ldots, v_{n-1},u_6,u_8,\ldots,u_{n-2}\}$, $V_{c_3}=\{v_4,u_1,u_2\}$, $V_{c_4}=\{v_3,v_n\}$ and $V_{c_5}=\{v\}$. Here,  $\theta(c_1)=1+(\frac{n-8}{2}+1)+(\frac{n-4}{2}+1)+2=(n-1)$, $\theta(c_2)=1+(\frac{n-6}{2}+1)+(\frac{n-8}{2}+1)=(n-4)$, $\theta(c_3)=3$, $\theta(c_4)=2$ and $\theta(c_5)=1$. Therefore, the corresponding $b$-chromatic sum is given by$\varphi'(H_n)=1(n-1)+2(n-4)+3\times 3+4\times 2+5\times 1=3n+13$.

\item let $n$ be odd. Here, we can colour the vertices of $H_n$ such that we get the colour classes $V_{c_1}=\{v_1,v_6,v_8,\ldots,v_{n-1},u_4, u_3,u_5,u_7,\ldots, u_{n-2},u_n\}$, $V_{c_2}=\{v_2, v_5,v_7,v_{9},\ldots, v_{n-2},u_6,u_8,\ldots,u_{n-1}\}$, $V_{c_3}=\{v_4,u_1,u_2\}$, $V_{c_4}=\{v_3,v_n\}$ and $V_{c_5}=\{v\}$. Here,  $\theta(c_1)=1+(\frac{n-7}{2}+1)+(\frac{n-3}{2}+1)+1=(n-1)$, $\theta(c_2)=1+(\frac{n-7}{2}+1)+(\frac{n-7}{2}+1)=(n-4)$, $\theta(c_3)=3$, $\theta(c_4)=2$ and $\theta(c_5)=1$. Therefore, the corresponding $b$-chromatic sum is given by $\varphi'(H_n)=1(n-1)+2(n-4)+3\times 3+4\times 2+5\times 1=3n+13$.
\end{enumerate}
\vspace{-0.25cm}
In both cases, we note that the colouring sum is $3n+13$. Therefore, $\varphi'(H_n)=3n+13$, if $n\ge 7$.
\end{enumerate}
\vspace{-0.5cm}
This completes the proof.
\end{proof}

In a similar way, we can find the $b^+$-chromatic sum of helm graphs as explained in the following theorem.

\begin{thm}\label{Thm-2.8}
The $b^+$-chromatic sum of a Helm graph $H_n$ is given by
\begin{equation*}
\varphi^+(H_n)=
\begin{cases}
21 & \text{if $n=3$},\\
29 & \text{if $n=4$},\\
34 & \text{if $n=5$},\\
48 & \text{if $n=6$},\\
9n-7 & \text{if $n\ge7$}.
\end{cases}
\end{equation*} 
\end{thm}
\begin{proof}
Let $v$ be the central vertex, $\{v_1,v_2,v_3,\ldots,v_n\}$ be the set of vertices in the cycle $C_n$ and $\{u_1,u_2,u_3,\ldots,u_n\}$ be the set of pendant vertices in $H_n$ as explained in the previous theorem. Then, consider the following cases.

\begin{enumerate}\itemsep0mm
\item[(i)] First assume that $n=3$. Colour the vertices of $H_3$ in such a way that the corresponding colour classes are $V_{c_1}=\{v\}$, $V_{c_2}=\{v_3\}$, $V_{c_3}=\{v_2,u_1\}$ and $V_{c_4}=\{v_1,u_2,u_3\}$. Therefore, $\varphi^+(H_3)=1\times 1+2\times 1+3\times 2+4\times 3=21$.

\item[(ii)] Now, assume that $n=4$. Colour the vertices of $H_4$ so that the corresponding colour classes are $V_{c_1}=\{v\}$, $V_{c_2}=\{v_1,u_3\}$, $V_{c_3}=\{v_2,u_4\}$, $V_{c_4}=\{v_3,u_1\}$ and $V_{c_5}=\{v_4,u_2\}$. Hence, $\varphi^+(H_4)=1\times 1+2\times 2+3\times 2+4\times 2+5\times 2=29$.

\item[(iii)] Next, assume that $n=5$.  Colour the vertices of $H_5$ in such a way that the corresponding colour classes are $V_{c_1}=\{v\}$, $V_{c_2}=\{v_3,v_5\}$, $V_{c_3}=\{v_2,u_1,u_4\}$ and $V_{c_4}=\{v_1,v_4,u_2,u_3,u_5\}$. Hence, $\varphi^+(H_5)=1\times 1+2\times 2+3\times 3+4\times 5=34$. 

\item[(iv)] Assume that $n=6$. A $b$-colouring of $H_6$ contains $5$ colours (see \cite{VS3}). Let $\cC=\{c_1,c_2,c_3,c_4,c_5\}$ be the required $b$-colouring of $H_6$. Here, we can colour the vertices of $H_6$ in such a way that the corresponding colour classes are $V_{c_1}=\{v\}$,  $V_{c_2}=\{v_3,u_1\}$, $V_{c_3}=\{v_2,u_5\}$, $V_{c_4}=\{v_4,v_6,u_2\}$ and $V_{c_5}=\{v_1,u_3,u_4,u_5,u_6\}$. Therefore, $\theta(c_1)=1,\theta(c_2)=\theta(c_3)=2, \theta(c_4)=3$ and $\theta(c_5)=5$. Hence, $\varphi^+(H_6)=1\times 1+2\times 2+3\times 2+4\times 3+5\times 5=48$.

\item[(v)] Assume that $n\ge 7$. Then, we need to consider the following cases. 
\begin{enumerate}\itemsep0mm
\item Let $n$ be even. Now, colour the vertices of $H_n$ such that we get the colour classes 
$V_{c_1}=\{v\}$, $V_{c_2}=\{v_3,v_n\}$,  $V_{c_3}=\{v_4,u_1,u_2\}$, $V_{c_4}=\{v_2, v_7,v_9,v_{11},\ldots, v_{n-1},u_6,u_8,\ldots,u_{n-2}\}$ and $V_{c_5}=\{v_1,v_6,v_8,\ldots,v_{n-2},\\u_3,u_5,u_7,\ldots, u_{n-1},u_4,u_n\}$.  Here,  $\theta(c_1)=1$, $\theta(c_2)=2$, $\theta(c_3)=3$, $\theta(c_4)=n-4$ and $\theta(c_5)=(n-1)$. Therefore, the corresponding $b^+$-chromatic sum is given by$\varphi^+(H_n)=1\times 1+2\times 2+3\times 3+4\times (n-4)+5\times (n-1)=9n-7$.

\item let $n$ be odd. Here, we can colour the vertices of $H_n$ such that colour classes are , $V_{c_5}=\{v\}$,$V_{c_2}=\{v_3,v_n\}$, $V_{c_3}=\{v_4,u_1,u_2\}$, $V_{c_4}=\{v_2, v_5,v_7,\\ v_{9},\ldots, v_{n-2},u_6,u_8,\ldots,u_{n-1}\}$ and $V_{c_5}=\{v_1,v_6,v_8,\ldots,v_{n-1},u_4, u_3,u_5,u_7,\\ \ldots, u_{n-2},u_n\}$. Here,  $\theta(c_1)=1$, $\theta(c_2)=2$, $\theta(c_3)=3$, $\theta(c_4)=n-4$ and $\theta(c_5)=n-1$. Therefore, the corresponding $b$-chromatic sum is given by $\varphi^+(H_n)=1\times 1+2\times 2+3\times 3+4(n-4)+5\times (n-1)=9n-7$.
\end{enumerate}
\vspace{-0.25cm}
In both cases, we note that the $b^+$-chromatic sum is $9n-7$. Therefore, $\varphi^+(H_n)=9n-7$, if $n\ge 7$.
\end{enumerate}
\vspace{-0.35cm}
This completes the proof.
\end{proof}

\subsection{Colouring Sums of Closed Helm Graphs}

A \textit{closed helm} is a graph obtained by forming an outer cycle joining the pendant vertices of a helm graph. A closed Helm graph is denoted by $CH_n$. It is to be noted that the coloring patterns in a minimal proper colouring of Helm graphs and closed Helm graphs are the same and the number of vertices having the same colour are also the same in both of these graphs. Hence, the following results are obvious.

\begin{thm}\label{Thm-2.9}
The $\chi$-chromatic sum of a closed Helm graph $CH_n$ is given by
\begin{equation*}
\chi'(CH_n)=
\begin{cases}
3(n+1); & \text{if $n$ is even},\\
3n+7; & \text{if $n$ is odd}.
\end{cases}
\end{equation*}
\end{thm}

\begin{thm}\label{Thm-2.10}
The $\chi^+$-chromatic sum of a Helm graph $H_n$ is given by
\begin{equation*}
\chi^+(CH_n)=
\begin{cases}
5n+1; & \text{if $n$ is even},\\
7n-2; & \text{if $n$ is odd}.
\end{cases}
\end{equation*}
\end{thm}

It is also proved in \cite{VS3} that a $b$-colouring of a Helm graph and a closed Helm graph are the same, but because of the difference in adjacency patterns in them, the $b$-chromatic sum and $b^+$-chromatic sums are different for these graphs. The following theorem discusses the $b$-chromatic sum of a closed Helm graph.

\begin{thm}\label{Thm-2.11}
The $b$-chromatic sum of a closed Helm graph $CH_n$ is given by
\begin{equation*}
\varphi'(CH_n)=
\begin{cases}
16; & \text{if $n=3$},\\
25; & \text{if $n=4$},\\
22; & \text{if $n=5$},\\
32; & \text{if $n=6$},\\
3(n+5); & \text{if $n>7$},\\
3n+16; & \text{if $n\ge 7$ and is odd}.
\end{cases}
\end{equation*}
\end{thm}
\begin{proof}
Let $v$ be the central vertex, $\{v_1,v_2,v_3,\ldots,v_n\}$ be the set of vertices in the inner cycle and $\{u_1,u_2,u_3,\ldots,u_n\}$ be the set of vertices of the outer cycle in the closed helm $CH_n$. Then we have the following cases.

\begin{enumerate}\itemsep0mm
\item[(i)] First assume that $n=3$. Let the required $b$-colouring of $CH_3$ be $\cC=\{c_1,c_2,c_3,c_4\}$. Colour the vertices of $CH_3$ in such a way that the corresponding colour classes are $V_{c_1}=\{v_1,u_2\}$, $V_{c_2}=\{v_2,u_3\}$, $V_{c_3}=\{v_3,u_1\}$ and $V_{c_4}=\{v\}$. This colouring is illustrated in Figure \ref{fig:Fig-8}. Therefore, $\theta(c_1)=\theta(c_2)=\theta(c_3)=2$ and $\theta(c_4)=1$. Hence, $\varphi'(CH_3)=1\times 2+2\times 2+3\times 2+4\times1=16$.

\begin{figure}[h!]
\centering
\includegraphics[width=0.65\linewidth]{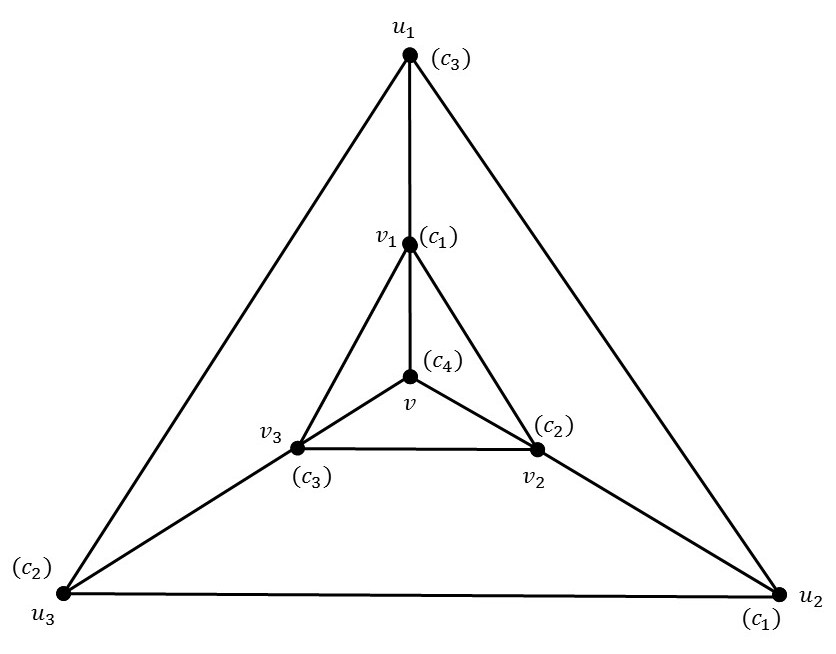}
\caption{\small A $b$-colouring of $CH_3$ so that the $b$-chromatic sum is minimum.}
\label{fig:Fig-8}
\end{figure}

\item[(ii)] Next, assume that $n=4$ and the required $b$-colouring be $\cC=\{c_1,c_2,c_3,c_4,c_5\}$. We can colour the vertices of $CH_4$ so that the corresponding colour classes are $V_{c_1}=\{v_1,u_3\}$, $V_{c_2}=\{v_2,u_4\}$, $V_{c_3}=\{v_4,u_2\}$, $V_{c_4}=\{v_3,u_1\}$ and $V_{c_5}=\{v\}$. This colouring is illustrated in Figure \ref{fig:Fig-9}. Therefore, $\theta(c_1)=\theta(c_2)=\theta(c_3)=\theta(c_4)=2$ and $\theta(c_5)=1$. Hence, $\varphi'(CH_4)=1\times 2+2\times 2+3\times 2+4\times 2+5\times 1=25$.

\begin{figure}[h!]
\centering
\includegraphics[width=0.55\linewidth]{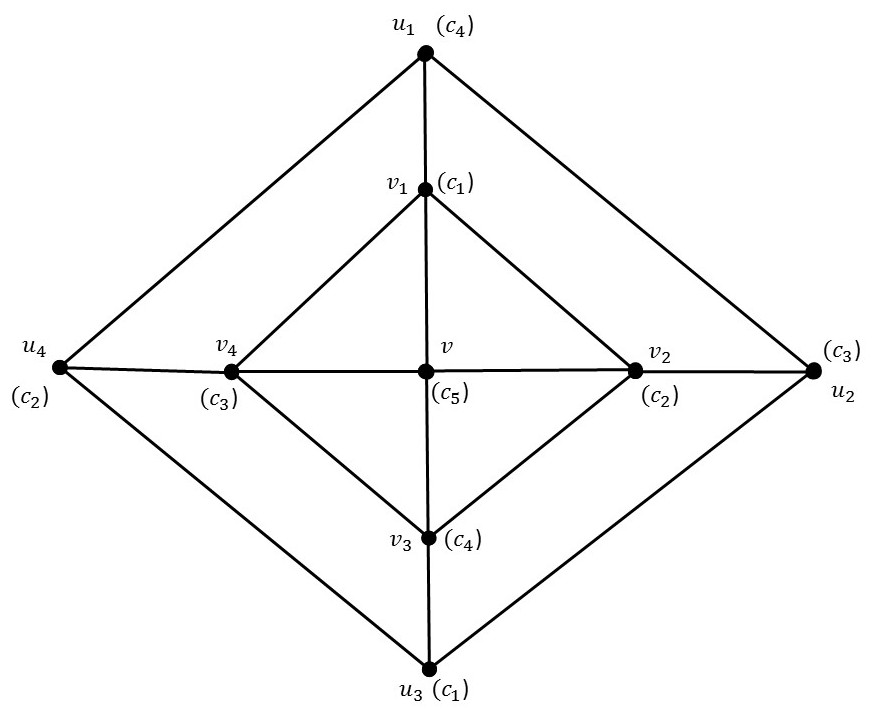}
\caption{\small A $b$-colouring of $CH_4$ so that the $b$-chromatic sum is minimum.}
\label{fig:Fig-9}
\end{figure}

\item[(iii)] Now, assume that $n=5$ and let the required $b$-colouring of $H_5$ be $\cC=\{c_1,c_2,c_3,c_4\}$. Colour the vertices of $CH_5$ in such a way that the corresponding colour classes are $V_{c_1}=\{v_1,v_3,u_2,u_5\}$, $V_{c_2}=\{v_2,v_4,u_1,u_3\}$, $V_{c_3}=\{v_5,u_4\}$,  and $V_{c_4}=\{v\}$. This colouring is illustrated in Figure \ref{fig:Fig-10}. Therefore, $\theta(c_1)=\theta(c_2)=4, \theta(c_3)=2$ and $\theta(c_4)=1$. Hence, $\varphi'(CH_5)=1\times 4+2\times 4+3\times 2+4\times 1=22$. 

\begin{figure}[h!]
\centering
\includegraphics[width=0.55\linewidth]{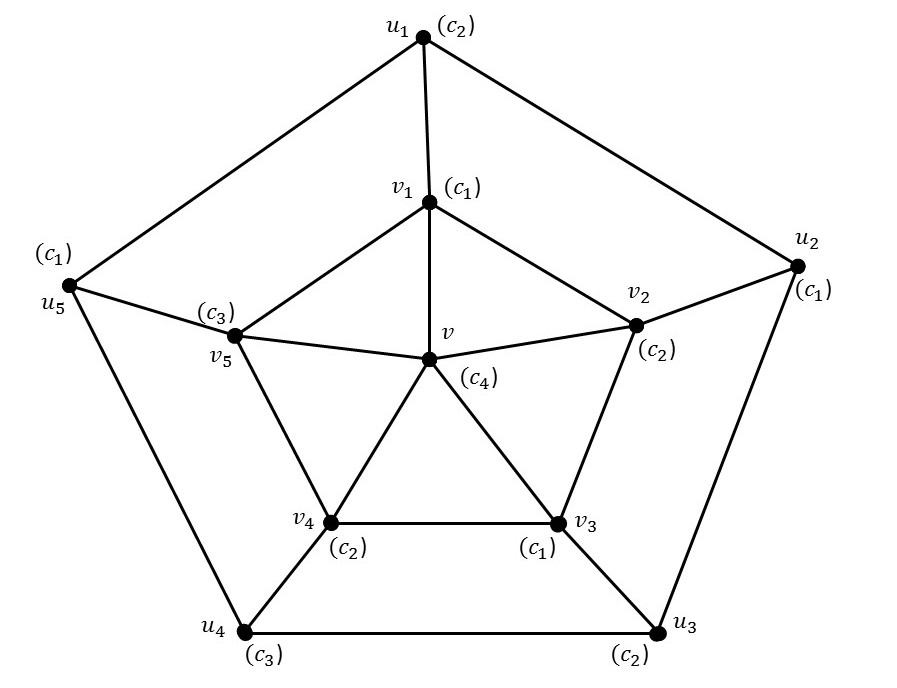}
\caption{\small A $b$-colouring of $CH_5$ so that the $b$-chromatic sum is minimum.}
\label{fig:Fig-10}
\end{figure}

\item[(iv)] Assume that $n=6$. Let the required $b$-colouring of $CH_6$ be $\cC=\{c_1,c_2,c_3,c_4,c_5\}$. Here, Colour the vertices of $CH_6$ in such a way that the corresponding colour classes are $V_{c_1}=\{v_1,v_5,u_3,u_4,u_6\}$, $V_{c_2}=\{v_4,v_6,u_2,u_5\}$, $V_{c_3}=\{v_2,u_4\}$, $V_{c_4}=\{v_3,u_1\}$ and $V_{c_5}=\{v\}$. This colouring is illustrated in Figure \ref{fig:Fig-11}. Therefore, $\theta(c_1)=\theta(c_2)=4, \theta(c_3)=\theta(c_4)=2$ and $\theta(c_5)=1$. Hence, $\varphi'(CH_6)=1\times 5+2\times 4+3\times 2+4\times 2+5\times 1=32$.

\begin{figure}[h!]
\centering
\includegraphics[width=0.55\linewidth]{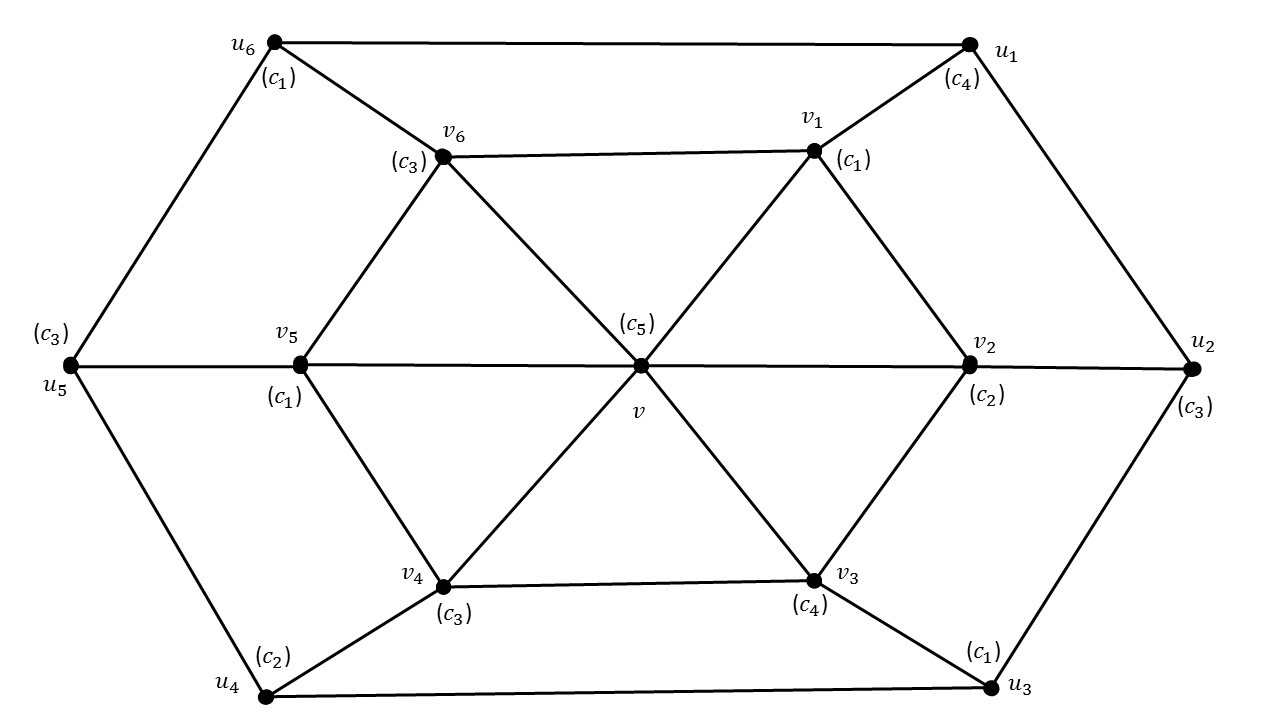}
\caption{\small A $b$-colouring of $CH_6$ so that the $b$-chromatic sum is minimum.}
\label{fig:Fig-11}
\end{figure}

\item[(v)] Assume that $n\ge 7$. Also, let $\cC=\{c_1,c_2,c_3,c_4,c_5\}$ be the required  $b$-coloring of the closed Helm graph $CH_n$. Then, the following cases are to be considered.
 
\begin{enumerate}\itemsep0mm
\item Let $n$ be even. Then, we can colour the vertices of $CH_n$ such that the colour classes are $V_{c_1}=\{v_1,v_5,v_7,\ldots,v_{n-1},u_3,u_6,u_8,\ldots, u_n\}$, $V_{c_2}=\{v_2, v_6,v_8,v_{10},\ldots, v_{n-2},u_4,u_7,u_9,u_11,\ldots,u_{n-1}\}$, $V_{c_3}=\{v_4,v_n,u_2,u_5\}$, $V_{c_4}=\{v_3,u_1\}$ and $V_{c_5}=\{v\}$. Here,  $\theta(c_1)=1+(\frac{n-6}{2}+1)+1+(\frac{n-6}{2}+1)=(n-2)$, $\theta(c_2)=1+(\frac{n-8}{2}+1)+1+(\frac{n-8}{2}+1)=(n-4)$, $\theta(c_3)=4$, $\theta(c_4)=2$ and $\theta(c_5)=1$. Therefore, the corresponding $b$-chromatic sum is given by $\varphi'(CH_n)=1(n-2)+2(n-4)+3\times 4+4\times 2+5\times 1=3n+15=3(n+5)$.

\item let $n$ be odd. Here, we can colour the vertices of $CH_n$ so that the colour classes are $V_{c_1}=\{v_1,v_5,v_7,\ldots,v_{n-2},u_3, u_6,u_8,u_{10},\ldots, u_{n-1},\}$, $V_{c_2}=\{v_2, v_6,v_8,v_{10},\ldots, v_{n-1},u_4,u_7,u_9,u_{11},\ldots,u_n\}$, $V_{c_3}=\{v_4,v_n,u_2,u_5\}$, $V_{c_4}=\{v_3,u_1\}$ and $V_{c_5}=\{v\}$. Here,  $\theta(c_1)=1+(\frac{n-7}{2}+1)+1+(\frac{n-7}{2}+1)=(n-3)$, $\theta(c_2)=1+(\frac{n-7}{2}+1)+1+(\frac{n-7}{2}+1)=(n-3)$, $\theta(c_3)=4$, $\theta(c_4)=2$ and $\theta(c_5)=1$. Therefore, the corresponding $b$-chromatic sum is given by $\varphi'(CH_n)=1(n-3)+2(n-3)+3\times 4+4\times 2+5\times 1=3n+16$.
\end{enumerate}
\end{enumerate}
\vspace{-0.25cm}
This completes the proof.
\end{proof}

Reversing the colouring pattern of the vertices of $CH_n$ in Theorem \ref{Thm-2.11}, as indicated in Theorem \ref{Thm-2.4} and Theorem \ref{Thm-2.8}, we can determine the $b^+$-chromatic sum of a closed Helm graph as follows.

\begin{thm}\label{Thm-2.12}
The $b^+$-chromatic sum of a closed Helm graph $CH_n$ is given by
\begin{equation*}
\varphi^+(CH_n)=
\begin{cases}
19; & \text{if $n=3$},\\
29; & \text{if $n=4$},\\
33; & \text{if $n=5$},\\
47; & \text{if $n=6$},\\
9(n-1); & \text{if $n>7$ and $n$ is even},\\
9n-10; & \text{if $n\ge 7$ and $n$ is odd}.
\end{cases}
\end{equation*}
\end{thm}

\subsection{Colouring Sums of Sunlet Graphs}

Another class of graph related to cycles is the class of sunlet graphs. A \textit{sunlet graph}, denoted by $S_n$, is defined as $S_n= C_n\odot K_1$, where $\odot$ represents the corona product of two graphs. Or in other words, a sunlet graph is the graph obtained by attaching a pendant edge to each vertex of a cycle $C_n$.

\ni The $\chi$-chromatics sum of a sunlet graph is determined in the following result. 

\begin{prop}\label{Prop-2.13}
The $\chi$-chromatic sum of a sunlet graph $S_n$ is given by 
\begin{equation*}
\chi'(S_n)=
\begin{cases}
3n; & \text{if $n$ is even},\\
3(n+1); & \text{if $n$ is odd}.
\end{cases}
\end{equation*}
\end{prop}
\begin{proof}
Let $\{v_1,v_2,v_3,\ldots,v_n\}$ be the set of vertices of $C_n$ and $u_1,u_2,u_3,\ldots,u_n$ be the pendant vertices of the sunlet graph $S_n$. Then we have the following cases.
\begin{enumerate}
\item[(i)] If $n$ is even, then $S_n$ is bipartite and hence is $2$-colourable. Let $c_1$ and $c_2$ be the two colours used for colouring the vertices of $S_n$. By colouring the vertices of $S_n$ suitably, we get the corresponding colour classes $V_{c_1}=\{v_1,v_3,v_5,\ldots,v_{n-1}, \\ u_2,u_4, u_6, \ldots, u_n\}$ and $V_{c_2}=\{v_2,v_4,v_6,\ldots,v_n, u_1,u_3,u_5, \ldots, u_{n-1}\}$. Hence, we have  $\theta(c_1)=\theta(c_2)=n$ and hence $\chi'(S_n)=1(n)+2(n)=3n$.

\item[(ii)] If $n$ is odd, then $S_n$ is $3$-colourable, using the colours, say $c_1, c_2$ and $c_3$. By colouring the vertices of $S_n$ suitably, we get the colour classes $V_{c_1}=\{v_1,v_3,v_5,\ldots, v_{n-2}, u_2,u_4,u_6, \ldots, u_{n-1}\}$, $V_{c_2}=\{v_2,v_4,v_6,\ldots,v_{n-1}, u_1,u_3,u_5, \\ \ldots, u_{n-2}\}$ and $V_{c_3}=\{v_n,u_1\}$. Therefore, $\theta(c_1)=\theta(c_2)=n-1$ and $\theta(c_3)=2$. Hence, $\chi'(S_n)=1(n-1)+2(n-1)+3\times 2=3(n+1)$.
\end{enumerate}
\vspace{-0.3cm}
This completes the proof.
\end{proof}

% % % % % % % % % % % % % % % % % % % % % % % % % % % % % % %

The $\chi^+$-chromatic number of $S_n$ can be calculated by reversing the colouring pattern of the vertices of $S_n$ mentioned in Proposition \ref{Prop-2.13} and hence we get the following result immediately.

\begin{prop}\label{Thm-2.14}
The $\chi^+$-chromatic sum of a sunlet graph $S_n$ is given by
\begin{equation*}
\chi^+(S_n)=
\begin{cases}
3n; & \text{if $n$ is even},\\
5n-3; & \text{if $n$ is odd}.
\end{cases}
\end{equation*}
\end{prop}

\ni Next, we determine the $b$-chromatic sum of the sunlet graphs $S_n$ in the following theorem.

\begin{thm}\label{Thm-2.15}
The $b$-chromatic sum of a sunlet graph $S_n$ is given by 
\begin{equation*}
\varphi'(G)=
\begin{cases}
10; & \text{if $n=3$},\\
20; & \text{if $n=4$},\\
17; & \text{if $n=5$},\\
3n+8; & \text{if $n\ge 6$}.
\end{cases}
\end{equation*}
\end{thm}
\begin{proof}
Let $\{v_1,v_2,v_3,\ldots,v_n\}$ be the set of vertices of $C_n$ and $u_1,u_2,u_3,\ldots,u_n$ be the pendant vertices of the sunlet graph $S_n$. Then we have the following cases.

\begin{enumerate}\itemsep0mm
\item[(i)] First assume that $n=3$. It is known that a $b$-coloring of the sun graph $S_3$ contains $3$ colours (see \cite{VV1}). Let this colouring be $\cC=\{c_1,c_2,c_3\}$. Colour the vertices of $S_3$ in such a way that the corresponding colour classes we get are $V_{c_1}=\{v_1,u_2,u_3\}$, $V_{c_2}=\{v_2,u_1\}$ and $V_{c_3}=\{v_3\}$. This colouring is illustrated in Figure \ref{fig:Fig-5}. Therefore, $\theta(c_1)=3,\theta(c_2)=2$ and $\theta(c_3)=1$. Hence, $\varphi'(S_3)=1\times 3+2\times 2+3\times 1=10$.

\begin{figure}[h!]
\centering
\includegraphics[width=0.6\linewidth]{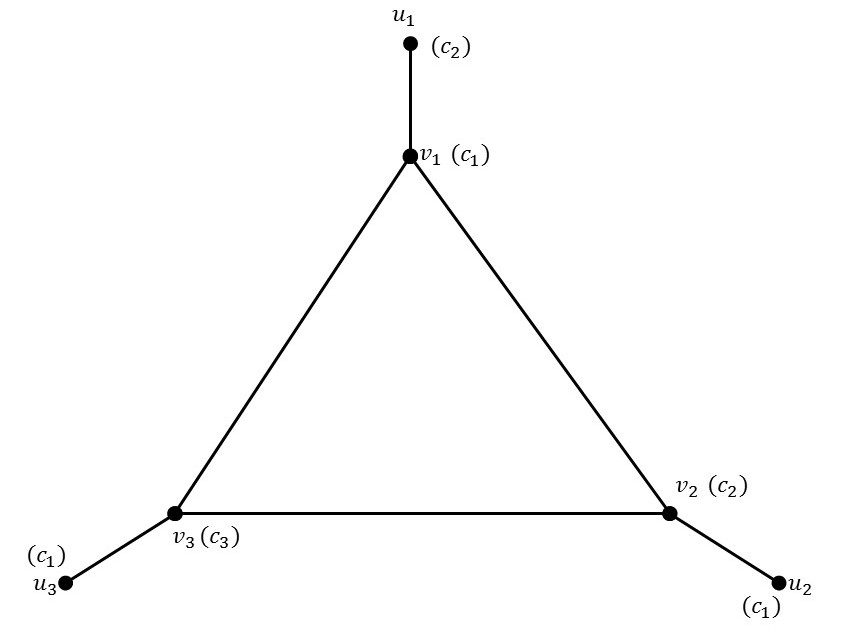}
\caption{\small A $b$-colouring of $S_3$ so that the $b$-chromatic sum is minimum.}
\label{fig:Fig-5}
\end{figure}

\item[(ii)] Next, assume that $n=4$. The $b$-colouring of $S_4$ contains $4$ colours (see \cite{VV1}). Let this colouring be $\cC=\{c_1,c_2,c_3,c_4\}$. We can colour the vertices so that the corresponding colour classes are $V_{c_1}=\{v_1,u_3\}$, $V_{c_2}=\{v_2,u_4\}$, $V_{c_3}=\{v_3,u_1\}$ and $V_{c_4}=\{v_4,u_2\}$. This colouring is illustrated in Figure \ref{fig:Fig-6}. Therefore, $\theta(c_1)=\theta(c_2)=\theta(c_3)=\theta(c_4)=2$. Hence, $\varphi'(S_4)=1\times 2+2\times 2+3\times 2+4\times 2=20$.

\begin{figure}[h!]
\centering
\includegraphics[width=0.6\linewidth]{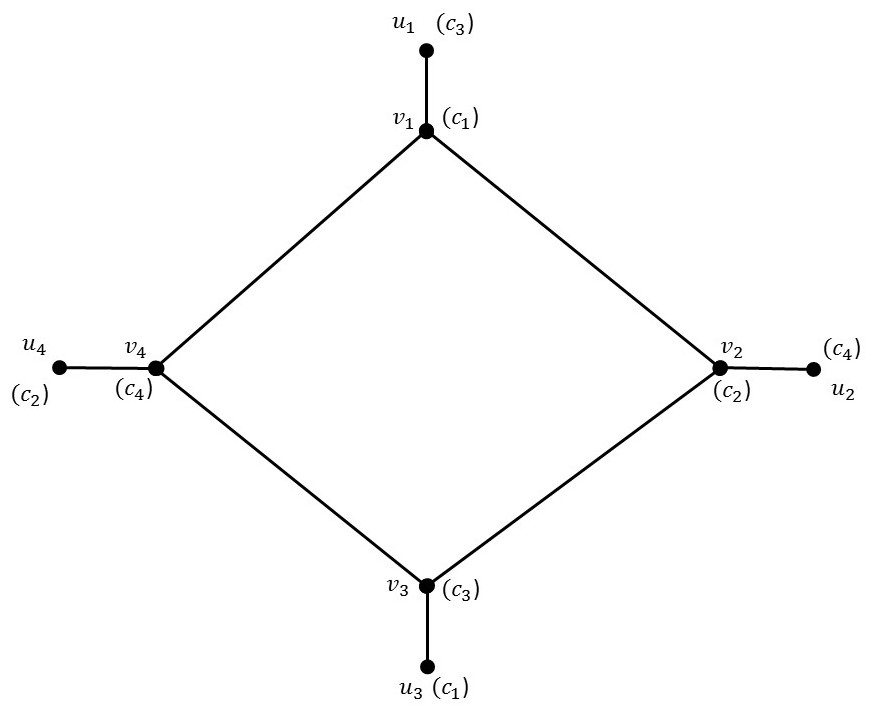}
\caption{\small A $b$-colouring of $S_4$ so that the $b$-chromatic sum is minimum.}
\label{fig:Fig-6}
\end{figure}

\item[(iii)] Now, assume that $n=5$.  A $b$-colouring of $S_5$ contains $3$ colours (see \cite{VV1}). Let $\cC=\{c_1,c_2,c_3\}$ be such a $b$-colouring of $S_5$. We can colour the vertices of $S_5$ in such a way that the corresponding colour classes are $V_{c_1}=\{v_1,v_4,u_2,u_3,u_5\}$, $V_{c_2}=\{v_2,u_1,u_4\}$ and  $V_{c_3}=\{v_3,v_5\}$. This colouring is illustrated in Figure \ref{fig:Fig-7}. Therefore, $\theta(c_1)=5, \theta(c_2)=3$  and  $\theta(c_3)=2$. Hence, $\varphi'(H_5)=1\times 5+2\times 3+3\times 2=17$. 

\begin{figure}[h!]
\centering
\includegraphics[width=0.6\linewidth]{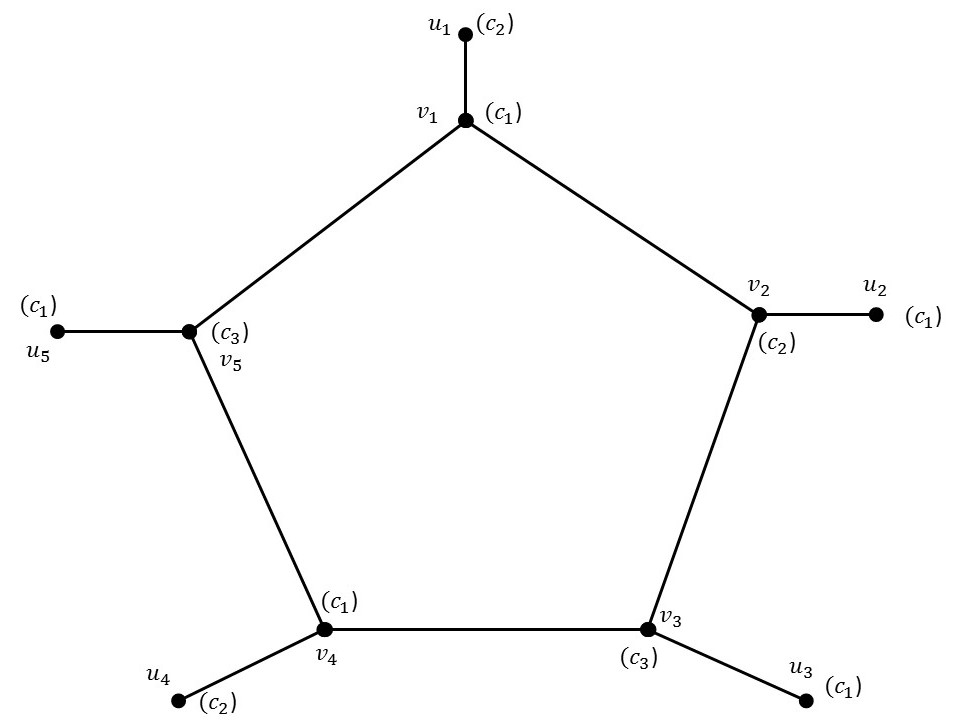}
\caption{\small A $b$-colouring of $S_5$ so that the $b$-chromatic sum is minimum.}
\label{fig:Fig-7}
\end{figure}

\item[(v)] Assume that $n\ge 6$. A $b$-coloring of a sunlet graph $S_n;n\ge 6$ contains $4$ colours, say $c_1,c_2,c_3,c_4$. Then, we need to consider the following cases.
 
\begin{enumerate}\itemsep0mm
\item Let $n$ be even. Now, we can colour the vertices of $S_n$ such that we get the colour classes $V_{c_1}=\{v_1,v_6,v_8,\ldots,v_{n-2},u_3,u_5,u_7,\ldots, u_{n-1},u_4,u_n\}$, $V_{c_2}=\{v_2,v_5, v_7,v_9,\ldots, v_{n-1},u_6,u_8,\ldots,u_{n-2}\}$, $V_{c_3}=\{v_4,u_1,u_2\}$ and $V_{c_4}=\{v_3,v_n\}$. Here,  $\theta(c_1)=1+(\frac{n-8}{2}+1)+(\frac{n-4}{2}+1)+2=(n-1)$, $\theta(c_2)=1+(\frac{n-6}{2}+1)+(\frac{n-8}{2}+1)=(n-4)$, $\theta(c_3)=3$ and $\theta(c_4)=2$. Therefore, the corresponding $b$-chromatic sum is given by$\varphi'(S_n)=1(n-1)+2(n-4)+3\times 3+4\times 2=3n+8$.

\item let $n$ be odd. Here, we can colour the vertices of $S_n$ such that we get the colour classes $V_{c_1}=\{v_1,v_6,v_8,\ldots,v_{n-1},u_4, u_3,u_5,u_7,\ldots, u_{n-2},u_n\}$, $V_{c_2}=\{v_2, v_5,v_7,v_{9},\ldots, v_{n-2},u_6,u_8,\ldots,u_{n-1}\}$, $V_{c_3}=\{v_4,u_1,u_2\}$ and $V_{c_4}=\{v_3,v_n\}$. Here,  $\theta(c_1)=1+(\frac{n-7}{2}+1)+(\frac{n-3}{2}+1)+1=(n-1)$, $\theta(c_2)=1+(\frac{n-7}{2}+1)+(\frac{n-7}{2}+1)=(n-4)$, $\theta(c_3)=3$ and $\theta(c_4)=2$. Therefore, the corresponding $b$-chromatic sum is given by $\varphi'(S_n)=1(n-1)+2(n-4)+3\times 3+4\times 2=3n+8$.
\end{enumerate}
\vspace{-0.25cm}
In both cases, we note that the colouring sum is $3n+8$. Therefore, $\varphi'(S_n)=3n+8$, if $n\ge 6$.
\end{enumerate}
\vspace{-0.5cm}
This completes the proof.
\end{proof}

Colouring the vertices of $S_n$ in the reverse order of that mentioned in Theorem \ref{Thm-2.10}, we can determine the $b^+$-chromatic sum of sunlet graphs $S_n$ as follows.

\begin{thm}\label{Thm-2.16}
The $b^+$-chromatic sum of a sunlet graph $S_n$ is given by 
\begin{equation*}
\varphi^+(G)=
\begin{cases}
14; & \text{if $n=3$},\\
20; & \text{if $n=4$},\\
23; & \text{if $n=5$},\\
7n-8; & \text{if $n\ge 6$}.
\end{cases}
\end{equation*}
\end{thm}

\subsection{Colouring Sums of Web Graphs}

A \textit{web graph}, denoted by $W_n$, is the graph obtained by attaching a pendant edge to each vertex of one cycle (say outer cycle) of the prism graph $C_n\Box P_2$, where $\Box$ represents the Cartesian product of two graphs.

Similar to Proposition \ref{Prop-2.13}, the following proposition discusses the $chi$-chromatic sum of a web graph $W_n$.

\begin{prop}\label{Prop-2.17}
The $\chi$-chromatic sum of a web graph $W_n$ is given by 
\begin{equation*}
\chi'(W_n)=
\begin{cases}
\frac{7n}{2}; & \text{if $n$ is even},\\
\frac{9n+9}{2}; & \text{if $n$ is odd}.
\end{cases}
\end{equation*}
\end{prop}
\begin{proof}
Let $\{v_1,v_2,v_3,\ldots,v_n\}$ be the set of vertices of of the inner cycle $C_n$ and $u_1,u_2,u_3,\ldots,u_n$ be the set of vertices of the outer cycle and $w_1,w_2,w_3,\ldots,w_n$ be the pendant vertices of the web graph $W_n$. Then, we have the following cases.

\begin{enumerate}\itemsep0mm

\item[(i)] If $n$ is even, then $W_n$ is a bipartite graph and hence is $2$-colourable. Let $c_1$ and $c_2$ be the two colours used for colouring the vertices of $W_n$. By colouring the vertices of $S_n$ suitably, we get the corresponding colour classes $V_{c_1}=\{v_1,v_3,v_5,\ldots,v_{n-1},  u_2,u_4, u_6, \ldots, u_n, w_1,w_3,w_5,\ldots,w_{n-1}\}$ and $V_{c_2}=\{v_2,v_4,v_6,\ldots,v_n, u_1,u_3,u_5, \ldots, u_{n-1}, w_2,w_4,w_6,\ldots,w_n\}$. Therefore, we have   $\theta(c_1)=\theta(c_2)=\frac{3n}{2}$ and hence $\chi'(S_n)=1\times\frac{3n}{2}+2\times \frac{3n}{2}=\frac{7n}{2}$.

\item[(ii)] If $n$ is odd, then the web graph $W_n$ is also $3$-colourable. Let $\cC=\{c_1, c_2,c_3\}$ be the required $3$-colouring of $W_n$. By colouring the vertices of $W_n$ suitably, we get the colour classes $V_{c_1}=\{v_1,v_3,v_5,\ldots, v_{n-2}, u_2,u_4,u_6, \ldots, u_{n-1}, w_1,w_3,w_5,\\ \ldots,  w_{n-2}\}$, $V_{c_2}=\{v_2,v_4,v_6,\ldots,v_{n-1}, u_1,u_3,u_5, \ldots, u_{n-2},w_2,w_4,w_6,\ldots,w_{n-1}\}$ and $V_{c_3}=\{v_n,u_1,w_n\}$. Therefore, $\theta(c_1)=\theta(c_2)=\frac{3(n-1)}{2}$ and $\theta(c_3)=3$. Hence, $\chi'(W_n)=1\times \frac{3(n-1)}{2})+2\times \frac{3(n-1)}{2}+3\times 3=\frac{9n+9}{2}$.
\end{enumerate}
\vspace{-0.25cm}
This completes the proof.
\end{proof}

The $\chi^+$-chromatic number of the web graphs $W_n$ can be calculated by reversing the colouring pattern of the vertices of $W_n$ mentioned in Proposition \ref{Prop-2.17} and hence we get the following result immediately.

\begin{prop}\label{Thm-2.18}
The $\chi^+$-chromatic sum of a web graph $W_n$ is given by
\begin{equation*}
\chi^+(W_n)=
\begin{cases}
\frac{7n}{2}; & \text{if $n$ is even},\\
\frac{15n-9}{2}; & \text{if $n$ is odd}.
\end{cases}
\end{equation*}
\end{prop}

\ni The $b$-chromatic number of web graphs is determined in the following theorem.

\begin{thm}\label{Thm-2.19}
The $b$-chromatic number of a web graph $W_n$ is given by 
\begin{equation*}
\varphi(G)=
\begin{cases}
4 & \text{if $n=3,4$},\\
5 & \text{if $n\ge 5$}.
\end{cases}
\end{equation*}
\end{thm}
\begin{proof}
Let $V_1=\{v_1,v_2,v_3,\ldots,v_n\}$ be the set of vertices of the inner cycle $C_n$ and $V_2=\{u_1,u_2,u_3,\ldots,u_n\}$ be the set of vertices of the outer cycle and $V_3=\{w_1,w_2,w_3,\ldots,w_n\}$ be the set of pendant vertices of the web graph $W_n$. Then, we have the following cases.

\begin{enumerate}\itemsep0mm
\item[(i)] Let $n=3$. Then, colour the vertices of $W_3$ such that $c_1\mapsto v_1, c_2\mapsto v_2, c_3\mapsto v_3, c_4\mapsto u_1, c_4\mapsto u_1, c_3\mapsto u_2, c_2\mapsto u_3$ and $ c_1\mapsto \{w_1,w_2,w_3\}$. In this colouring, the vertex $v_1$ with colour $c_1$ is adjacent to the vertices $v_2$ with colour $c_2$, $v_3$ with colour $c_3$ and $u_1$ with colour $c_4$, the vertex $u_3$ with colour is adjacent to the vertices $w_3$ with colour $c_1$, $v_3$ with colour $c_3$ and $u_1$ with colour $c_4$ the vertex $u_2$ with colour $c_3$ is adjacent to the vertices $w_2$ with colour $c_1$, $u_3$ with colour $c_2$ and  $u_1$ with colour $c_4$ and the vertex $u_1$ with colour $c_4$ is adjacent to the vertices $v_1$ with colour $c_1$, $u_3$ with colour $c_2$ and $u_2$ with colour $c_3$ (see Figure \ref{fig:Fig-12}). More over, we have $\Delta(W_3)+1=4$ and hence this colouring is a $b$-colouring for $W_3$. Therefore, $\varphi(W_3)=4$.

\item[(ii)] Let $n=4$. Then, colour the vertices of $W_4$ in such a way that we have $c_1\mapsto v_1, c_2\mapsto v_2, c_3\mapsto v_3, c_4\mapsto v_4$, $c_1\mapsto u_3, c_2\mapsto u_4, c_3\mapsto v_2, c_4\mapsto u_1$, $c_1\mapsto\{w_1,w_2,w_4\}$ and $c_2\mapsto u_3$ (see Figure \ref{fig:Fig-13}). As explained in the above case, we can verify that this colouring is a $b$-colouring for $W_4$. Therefore, $\varphi(W_4)=4$.

\item[(iii)] Let $n\ge 5$. If $n=5$, we can colour the vertices of $W_5$ in such a way that $c_1\mapsto v_3, c_2\mapsto v_4, c_3\mapsto v_5, c_4\mapsto v_1, c_5 \mapsto v_2$, $c_1\mapsto u_1, c_2\mapsto u_2, c_3\mapsto u_3, c_4\mapsto u_4, c_5 \mapsto u_5$, $c_1\mapsto w_4, c_2\mapsto w_5, c_3\mapsto w_1, c_4\mapsto w_2, c_5 \mapsto w_3$ (see Figure \ref{fig:Fig-14}). It can be easily verified that it is a $b$-colouring for $W_5$ and hence $\varphi(W_5)=5$.

If $n>5$, then we can label the first five vertices of $V_1, V_2$ and $V_3$ as explained in the previous case, so that at least one vertex in each colour class is adjacent to some vertices in all other colour classes. Since  $\Delta(W_n)+1=5$, for all integral values of $n$, there exists no colouring for $W_n$ with more colours than this colouring. Hence, $\varphi(W_n)=5; n\ge 5$.
\end{enumerate}
\vspace{-0.25cm}
This completes the proof.
\end{proof}

\ni Now, we are in a position to discuss the $b$-chromatic sum of the web graphs. The following theorem discusses the $b$-chromatic sum of web graphs.

\begin{thm}\label{Thm-2.20}
The $b$-chromatic sum of a web graph $W_n$ is given by 
\begin{equation*}
\varphi'(W_n)=
\begin{cases}
18; & \text{if $n=3$},\\
25; & \text{if $n=4$},\\
45; & \text{if $n=5$},\\
5n+21; & \text{if $n\ge 6$ and $n$ is even},\\
5n+18; & \text{if $n\ge 7$ and $n$ is odd},\\
\end{cases}
\end{equation*}
\end{thm}
\begin{proof}
Let $V_1=\{v_1,v_2,v_3,\ldots,v_n\}$ be the set of vertices of of the inner cycle $C_n$ and $V_2=\{u_1,u_2,u_3,\ldots,u_n\}$ be the set of vertices of the outer cycle and $V_3=\{w_1,w_2,w_3,\ldots,w_n\}$ be the set of pendant vertices of the web graph $W_n$. Then, we have the following cases.

\begin{enumerate}\itemsep0mm
\item[(i)] First assume that $n=3$. The $b$-colouring of $W_3$ mentioned in Theorem \ref{Thm-2.19} yields the minimum possible colouring sum, for which the colour classes are $V_{c_1}=\{v_1,w_1,w_2,w_3\}, V_{c_2}=\{v_2,u_3\}, V_{c_3}=\{v_3,u_2\}$ and $V_{c_4}=\{u_1\}$. The colouring pattern mentioned here is illustrated in Figure \ref{fig:Fig-12}. Therefore, $\varphi'(W_3)=1\times 4+2\times2+3\times2+4\times1=18$.
 
\begin{figure}[h!]
\centering
\includegraphics[width=0.6\linewidth]{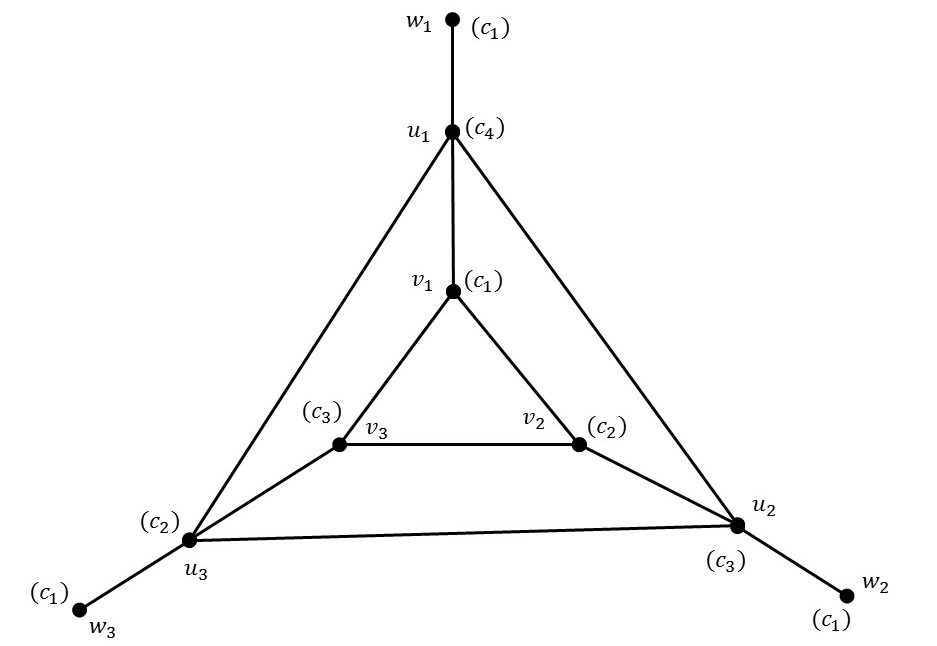}
\caption{\small A $b$-colouring of $W_3$ so that the $b$-chromatic sum is minimum.}
\label{fig:Fig-12}
\end{figure}

\item[(ii)] Assume that $n=4$. The $b$-colouring of $W_4$ mentioned in Theorem \ref{Thm-2.19} yields the minimum possible colouring sum. In this case, the corresponding colour classes are $V_{c_1}=\{v_1,u_3,w_1,w_2,w_4\}, V_{c_2}=\{v_2,u_4,w_3\}, V_{c_3}=\{v_4,u_2\}$ and $V_{c_4}=\{v_3,u_1\}$. The colouring pattern mentioned here is illustrated in Figure \ref{fig:Fig-13}. Therefore, $\varphi'(W_4)=1\times 5+2\times3+3\times2+4\times2=25$.

\begin{figure}[h!]
\centering
\includegraphics[width=0.6\linewidth]{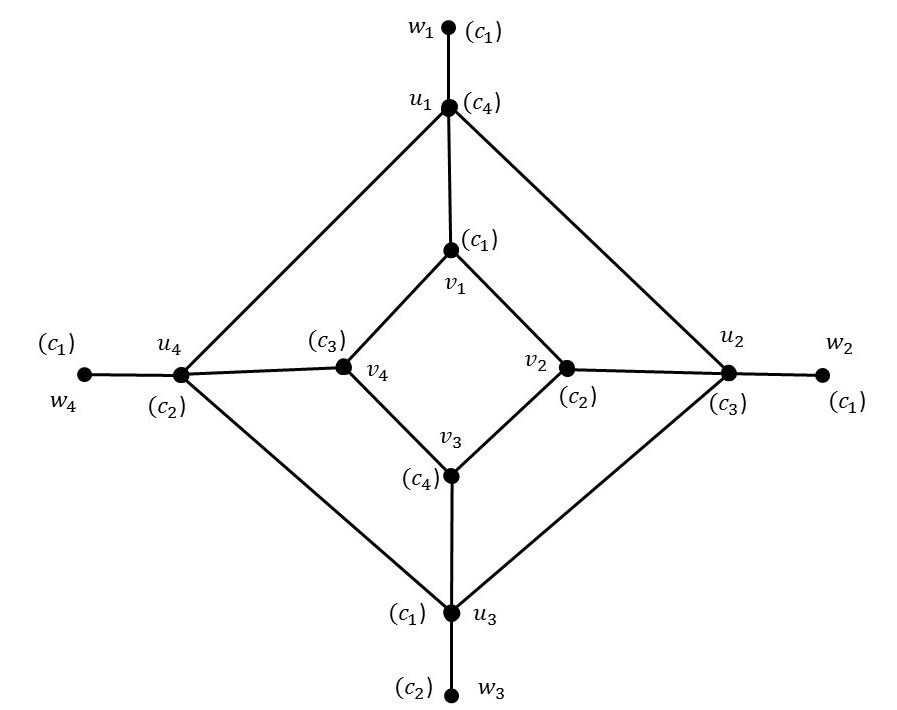}
\caption{\small A $b$-colouring of $W_4$ so that the $b$-chromatic sum is minimum.}
\label{fig:Fig-13}
\end{figure}

\item[(iii)] Now, assume that $n=5$.  Here also, the $b$-colouring of $W_5$ mentioned in Theorem \ref{Thm-2.19} yields the minimum possible colouring sum. In this case, the corresponding colour classes are $V_{c_1}=\{v_3,u_1,w_4\}, V_{c_2}=\{v_4,u_2,w_5\}, V_{c_3}=\{v_5,u_3,w_1\}$,  $V_{c_4}=\{v_1,u_4,w_2\}$ and $V_{c_5}=\{v_2,u_5,w_3\}$. This colouring pattern mentioned here is illustrated in Figure \ref{fig:Fig-14}. Therefore, $\varphi'(W_5)=1\times 3+2\times3+3\times3+4\times3+5\times 3=45$.

\begin{figure}[h!]
\centering
\includegraphics[width=0.7\linewidth]{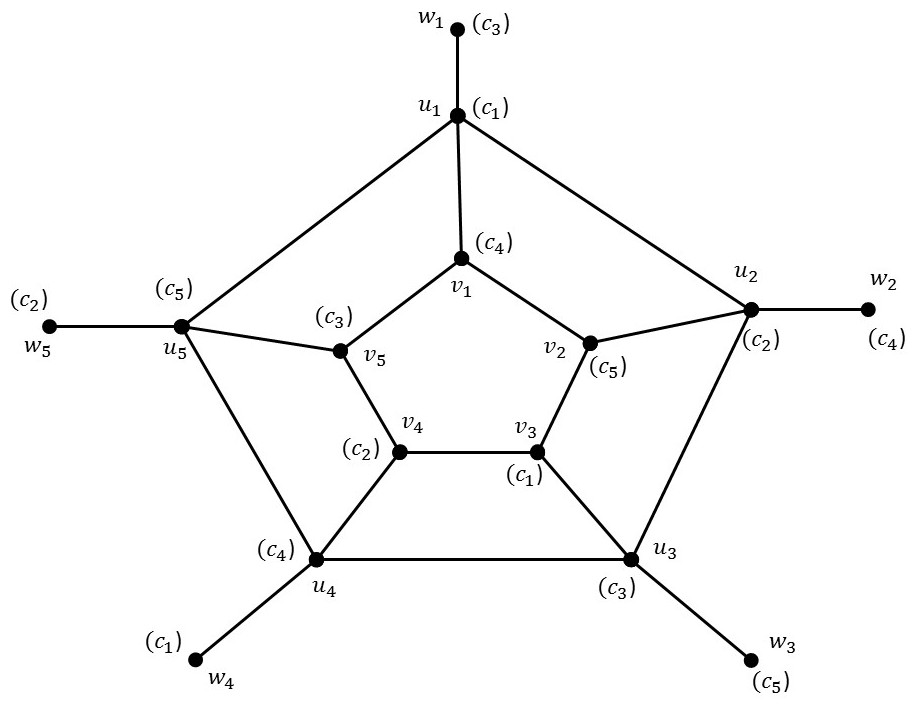}
\caption{\small A $b$-colouring of $W_5$ so that the $b$-chromatic sum is minimum.}
\label{fig:Fig-14}
\end{figure}

\item[(iv)] Next, assume that $n\ge 6$ is even. Here, we can colour the vertices of $W_n$ in such a way that the corresponding colour classes of $W_n$ are $V_{c_1}=\{v_3,v_7,v_9,\ldots,v_{n-1}, u_1, w_4, w_6,w_7,w_8, \ldots, w_n\}$, $V_{c_2}=\{v_4,v_6,v_8,\ldots,v_n,u_2, u_7,\\ u_9,\ldots, u_{n-1}, w_5\}$, $V_{c_3}=\{v_5,u_3,u_6,u_8,u_{10},\dots, u_n,w_1\}$, $V_{c_4}=\{v_1,u_4,w_2\}$ and $V_{c_5}=\{v_2,u_5,w_3\}$. Therefore, $\theta(c_1)=1+(\frac{n-8}{2}+1)+1+1+(n-6+1)=\frac{3n}{2}-5$, 
$\theta(c_2)=(\frac{n-4}{2}+1)+(\frac{n-8}{2}+1)+1=n-2$, $\theta(c_3)= 1+1+(\frac{n-6}{2}+1)=\frac{n}{2}+1$, $\theta(c_4)=\theta(c_5)=3$. Therefore, $\varphi'(W_n)= 1(\frac{3n}{2}-5)+2(n-2)+3(\frac{n}{2}+1)+4\times 3+5\times 3 =5n+21$.

\item[(v)] Now,  let $n\ge 7$ is odd. Here, we can colour the vertices of $W_n$ such that the corresponding colour classes are $V_{c_1}=\{v_3,v_6,v_8,\ldots,v_{n-1}, u_1, w_4,w_5,w_6,\ldots, w_n\}$, $V_{c_2}=\{v_2,v_7,v_9,\ldots,v_n,u_5, u_8,u_{10},u_{12}\ldots, u_{n-1}, w_1,w_3\}$, $V_{c_3}=\{v_5,u_3,u_7,u_9,\\  u_{11},\dots, u_n\}$, $V_{c_4}=\{v_1,u_4,w_2\}$ and $V_{c_5}=\{v_4,u_2,u_6\}$. Then, $\theta(c_1)=1+(\frac{n-7}{2}+1)+1+1+n-3=\frac{3n-7}{2}$,  $\theta(c_1)=1+(\frac{n-7}{2}+1)+1+1+(\frac{n-9}{2}+1)+2=(n-2)$, $\theta(c_3)=1+1+(\frac{n-7}{2}+1)+1=\frac{n-1}{2}$ and $\theta(c_4)=\theta(c_5)=3$. Hence, we have $\varphi'(W_n)=1(\frac{3n-7}{2})+2(n-2)+3(\frac{n-1}{2})+4\times 3+5\times 3=5n+18$.
\end{enumerate}
\vspace{0.25cm}
This completes the proof.
\end{proof}

The $b^+$-chromatic number of $W_n$ can be calculated by reversing the colouring pattern of the vertices of $W_n$ explained in \ref{Thm-2.20} and hence we get the following result straight forward.

\begin{thm}\label{Thm-2.21}
The $b^+$-chromatic sum of a web graph $W_n$ is given by 
\begin{equation*}
\varphi^+(W_n)=
\begin{cases}
27; & \text{if $n=3$},\\
35; & \text{if $n=4$},\\
45; & \text{if $n=5$},\\
13n-21; & \text{if $n\ge 6$ and $n$ is even},\\
13n-18; & \text{if $n\ge 7$ and $n$ is odd},\\
\end{cases}
\end{equation*}
\end{thm}

\section{Conclusion}

So far, we have discussed some colouring sums related to certain types of graph colouring for some cycle related graphs. The concepts discussed here can be applied or extended for other graph classes and graph operations. Some open problems related to this research area, we have identified during this study are the following.

\begin{prob}{\rm 
Discuss the general colouring sums of one point unions of cycles.}
\end{prob}

\begin{prob}{\rm 
Discuss the general colouring sums of unions of cycles with some edges in common.}
\end{prob} 

\begin{prob}{\rm 
Discuss the general colouring sums of join of two given cycles (and other arbitrary graphs).}
\end{prob}

\begin{prob}{\rm 
Discuss the general colouring sums of different graph powers of cycles (and other arbitrary graphs).}
\end{prob}

The concepts and results discussed in this paper may create some interests for further the research for some general colouring protocols. These facts highlight the wide scope of this area. 

\section*{Acknowledgements}
Authors of this article gratefully acknowledge the contributions of the anonymous reviewer whose critical and creative comments helped a lot in improving the quality and presentation style of this paper.

The first author of this article would like to dedicate this article to his teacher, mentor and motivator Prof. (Dr.) T. Thrivikraman, Professor Emeritus, Department of Mathematical Sciences, Kannur University, Kannur, India., as a tribute to his fifty years of service in teaching and research.


\begin{thebibliography}{99}

\bibitem{BST} A. Bar-Noy, H. Shachnai and T. Tamir, {\em On chromatic sums and distributed resource allocation}, in {\bf Proceedings of the Fourth Israel Symposium on Theory and Computing Systems}, 1996.

\bibitem{BM1} J. A. Bondy and U. S. R. Murty, \textbf {Graph theory with applications,} Macmillan Press, London, 1976. 

\bibitem{CZ1} G. Chartrand and P. Zhang, \textbf{Chromatic graph theory}, CRC Press, 2009.

\bibitem{EK1} B. Effatin and H. Kheddouci, \textit{The b-chromatic number of some power graphs}, Discrete Math. Theoret. Comput. Sci., \textbf{6}, (2003), 45-54.

\bibitem{GY1} J. L. Gross and J. Yellen, {\bf Handbook of graph Theory}, CRC Press, 2004.

\bibitem{FH} F. Harary, \textbf{Graph theory}, Addison Wesley, 1969.

\bibitem{IM1} R. W. Irving and D. F. Manlove, \textit{The b-chromatic number of a graph}, Discrete Appl. Math., \textbf{91}, (1999), 127-141.

\bibitem{AK1} A. Kohl, \textit{The $b$-chromatic number of power of cycles}, Discrete Math. Theoret. Comput. Sci., {\bf 15}(1)(2013); 147-156.

\bibitem{KS1} J. Kok and N. K. Sudev, \textit{The $b$-Chromatic Numbers of Certain Graphs and Digraphs}, J. Discrete Math. Sci. Cryptography., to appear.

\bibitem{KSC1} J. Kok, N. K. Sudev and K. P. Chithra, \textit{General colouring sums of graphs}, Cogent Math., \textbf{3}(1)(2016), 1-16, DOI: 10.1080/23311835.2016.1140002.

\bibitem{EK2} E. Kubicka, {\em The chromatic sum of a graph : History and recent developments}, Internat. J. Math. Math. Sci., {\bf 30}(2004), 1563-1573.

\bibitem{ES1} E. Kubicka and A. J. Schwenk, {\em An introduction to chromatic sums}, Proc. ACM Computer Science Conference, Louisville (Kentucky), 3945(1989).

\bibitem{LS1} P. C. Lisna and M. S. Sunitha, \textit{b-chromatic sum of a graph}, Discrete Math. Algorithms Appl., \textbf{7}(3), (2015), 1-15, DOI: 10.1142/S1793830915500408.

\bibitem{VI1} S. K. Vaidya and R. V. Isaac, \textit{The b-chromatic number of some path related graphs}, Int. J. Math. and Sci. Comput., \textbf{4}(1), (2014), pp 7-12.

\bibitem{VI2} S. K. Vaidya and R. V. Isaac, \textit{The b-chromatic number of some graphs}, Int. J. Math. Soft Comput., \textbf{5}(1), (2015), pp 165-169.

\bibitem{VS1} S. K. Vaidya and M. S. Shukla, \textit{The b-chromatic number of some cycle related graphs}, Int. J. Math. and Soft Comput., \textbf{4}(2), (2014), pp 113-127.

\bibitem{VS2} S. K. Vaidya and M. S. Shukla, \textit{The b-chromatic number of helm and closed helm}, Int. J. Math. Sci. Comput., {\bf 4}(2)(2014), 43-47.

\bibitem{VS3} S. K. Vaidya and M. S. Shukla, \textit{The b-chromatic number of wheel related graphs}, Malaya J. of Math., {\bf 2}(4)(2014), 482-488.

\bibitem{VV1}  J. V. Vivin and M. Vekatachalam, \textit{On $b$-chromatic number of sunlet graph and wheel graph families}, J. Egyptian Math. Soc., \textbf{23}(2), (2015), pp 215-222.

\bibitem{DBW} D. B. West, \textbf{Introduction to Graph Theory}, Pearson Education Inc., 2001.

\end{thebibliography}
\end{document}